\documentclass[11pt]{article}
\usepackage{amsmath,amssymb,amsthm}
\usepackage[usenames,dvipsnames]{color}
\usepackage{enumerate}
\usepackage{geometry}
\usepackage{graphicx}
\usepackage{fullpage}

\numberwithin{equation}{section}
\newtheorem{theorem}[equation]{Theorem}

\newtheorem{lemma}[equation]{Lemma}
\theoremstyle{definition}
\newtheorem{remark}[equation]{Remark}

\newtheorem{algorithm}{Algorithm}[section]

\newcommand{\dt}{{\Delta t}}

\newcommand{\norm}[1]{\left\Vert#1\right\Vert}

\newcommand{\R}{\mathbb{R}}

\usepackage[backgroundcolor=gray!30,linecolor=black]{todonotes}

\title{Simple and efficient continuous data assimilation of evolution equations via algebraic nudging}

\author{
Leo G. Rebholz\thanks{Department of Mathematical Sciences, Clemson University, Clemson, SC, 29634;
email: rebholz@clemson.edu, partially supported by NSF Grant DMS 1522191.}
\and
Camille Zerfas
\thanks{Department of Mathematical Sciences, Clemson University, Clemson, SC, 29634;
email: czerfas@clemson.edu, partially supported by NSF Grant DMS 1522191.}
}

\begin{document}
\date{}
\maketitle

\begin{abstract}
We introduce, analyze and test a new interpolation operator for use with continuous data assimilation (DA) of evolution equations that are discretized spatially with the finite element method.  The interpolant is constructed as an approximation of the $L^2$ projection operator onto piecewise constant functions on a coarse mesh, but which allows nudging to be done completely at the linear algebraic level, independent of the rest of the discretization, with a diagonal matrix that is simple to construct.  We prove the new operator
maintains stability and accuracy properties, and we apply it to algorithms for both fluid transport DA and incompressible Navier-Stokes DA.  For both applications we prove the DA solutions with arbitrary initial conditions converge to the true solution (up to optimal discretization error) exponentially fast in time, and are thus long-time accurate.  Results of several numerical tests are given, which both illustrate the theory and demonstrate its usefulness on practical problems.
\end{abstract}

\section{Introduction}

Data assimilation (DA) algorithms are widely used in weather prediction, climate modeling, and many other applications \cite{Kalnay_2003_DA_book}.  The term DA refers in general to schemes that incorporate observational data in simulations in order to increase accuracy and/or obtain better initial conditions.  There is a large amount of literature on the general topic of DA \cite{Daley_1993_atmospheric_book,Kalnay_2003_DA_book,Law_Stuart_Zygalakis_2015_book}, and several different techniques exist, such as the Kalman Filter, 3D/4D Var and others \cite{CHJ69,Daley_1993_atmospheric_book,Kalnay_2003_DA_book,Law_Stuart_Zygalakis_2015_book}.  Our interest herein
is with an approach recently pioneered by Azouani, Olson, and Titi \cite{Azouani_Olson_Titi_2014,Azouani_Titi_2014} (see also \cite{Cao_Kevrekidis_Titi_2001,Hayden_Olson_Titi_2011,Olson_Titi_2003}), which we call continuous DA.  This method adds a feedback control (penalty) term at the partial differential equation (PDE) level to nudge the computed solution toward the reference solution corresponding to observed data.  While this type of DA is similar to classical Newtonian relaxation {\color{black}methods \cite{Anthes_1974_JAS,Hoke_Anthes_1976_MWR},} the specific use of interpolation is a fundamental difference.  

The purpose of this paper is to introduce a new, simple, and accurate interpolation operator intended for use with DA algorithms for time dependent PDEs, where the DA is implemented with a nudging term, and a finite element spatial discretization is used.  The general form of such DA methods is
\[
v_t + N(v) + \mu I_H (v-u) = f,
\]
along with appropriate boundary conditions, where $I_H$ is an interpolation operator, $\mu>0$ is a nudging parameter, and $u$ is the true solution which can be partially observed so that $I_H(u)$ is considered known. In such methods, nudging of the computed solution is done by penalizing its difference to measurement data, i.e. $I_H(v)$ is penalized to be close to $I_H(u)$.   Since the initial work of Azouani, Olson and Titi in 2014 \cite{Azouani_Olson_Titi_2014}, there has been a large amount of work done for these types of methods, including for many different equations arising from physics such as Navier-Stokes and Boussinesq, for different variations of DA implementation, for noisy data, etc. \cite{Foias_Mondaini_Titi_2016,Bessaih_Olson_Titi_2015,Altaf_Titi_Knio_Zhao_Mc_Cabe_Hoteit_2015,
Larios_Pei_2017_KSE_DA_NL,Lunasin_Titi_2015,Albanez_Nussenzveig_Lopes_Titi_2016,Biswas_Martinez_2017,Farhat_Jolly_Titi_2015,Farhat_Lunasin_Titi_2016abridged,Farhat_Lunasin_Titi_2016benard,Foyash_Dzholli_Kravchenko_Titi_2014,GlattHoltz_Kukavica_Vicol_2014,Jolly_Martinez_Titi_2017,Larios_Lunasin_Titi_2015,Markowich_Titi_Trabelsi_2016,IMT18,LRZ18}.
Roughly speaking, in most of these papers it is proven that $v$ converges to $u$ exponentially fast in time (in an appropriate norm), for essentially any initial condition $v_0$, provided $\mu$ is in a particular range of values which depends on the coarse mesh width.

Despite all the recent work on continuous DA for various evolution equations, {\color{black} to date} there has been almost no work done for DA algorithms where finite element (FE) spatial discretizations are used.  To our knowledge, the papers \cite{IMT18,LRZ18,GNT18} are the only ones to consider this subject, and so far there is essentially no literature on best ways to implement continuous DA in the FE setting so that its use is feasible for large scale problems and with legacy codes.  In particular, in the FE setting, a fine mesh is needed to solve the FE problem, and a coarse mesh is needed for the implementation of nudging towards observation points.  While the use of multiple meshes is not uncommon, it leads to much more complicated programming and often does not (reasonably) allow for use with legacy codes.

The motivation for this work is to enable {\it simple} (but still effective) implementation of continuous DA into existing FE codes for evolution equations.  We will construct a new, simple, efficient, and effective interpolation operator for use with DA that can be implemented completely at the linear algebraic level, without any changes to the rest of the discretization.  We will show that for a given discretization of an evolution equation (assume here backward Euler time stepping for simplicity) which at each time step yields the linear algebraic system
\[
\left( \frac{1}{\Delta t}M + A\right) \hat{v}^{n+1} = \frac{1}{\Delta t}M\hat v^n + \hat f,
\]
the nudging from continuous DA with our new interpolation operator can be applied using an easily constructed diagonal matrix $D$ via
\[
\left( \frac{1}{\Delta t}M + A + \mu D \right) \hat{v}^{n+1} = \frac{1}{\Delta t}M\hat v^n + \hat f + \mu D \hat{u}^{n+1},
\]
with nonzero entries in $D$ occurring only where measurements of the true solution $u$ are taken, and $\hat{u}_j^{n+1} = u(x_j,t^{n+1})$ representing an observation {\color{black} at the FE} mesh node $x_j$ at time $t^{n+1}$.  Despite its simplicity, we prove that this new interpolation operator still maintains important stability and accuracy properties, similar to those described in \cite{LRZ18,GNT18,IMT18,Azouani_Olson_Titi_2014}.  We will then apply it to two applications, fluids transport and incompressible Navier-Stokes equations (NSE); in both cases, the DA algorithms give excellent analytical and numerical results. 

This paper is organized as follows. Immediately below, we introduce some notation and mathematical preliminaries.  Section 2 constructs and  analyzes the new interpolation operator, proving the critical stability and accuracy results for it, and discussing the FE implementation of the associated nudging term used in DA algorithms.  In section 3, we apply the new interpolation operator to DA of the fluid transport problem, proving a convergence result of the numerical DA solution to the true solution, exponentially fast, for essentially any initial condition, and also showing results of numerical tests that show the method is highly effective.  Section 4 applies the new interpolation operator to DA of incompressible NSE, which again reveals excellent analytical and numerical results.  Finally, conclusions and future directions are discussed in section 5.

\subsection{Notation and Preliminaries}

We consider $\Omega \subset \R^d$, $d$=2 or 3, to be a bounded open domain. The $L^2(\Omega)$ norm and inner product will be denoted by $\| \cdot \|$ and $(\cdot, \cdot)$, respectively, and all other norms will be appropriately labeled with subscripts. 

The Poincar\'e inequality will be used throughout this paper: there exists a constant $C_P$ depending only on $\Omega$ such that 
\[
\| \phi \| \le C_P \| \nabla \phi \| \ \forall \phi \in H^1_0(\Omega).
\]

The following lemma is proven in \cite{LRZ18}, and is useful in our analysis.

\begin{lemma}\label{geoseries}
Suppose constants $r$ and $B$ satisfy $r>1$, $B\ge 0$.  Then if the sequence of real numbers $\{a_n\}$ satisfies 
\begin{align}
ra_{n+1} \le a_n + B, \label{sequence}
\end{align}
we have that
\[
a_{n+1} \le a_0\left(\frac{1}{r}\right)^{n+1}  + \frac{B }{r-1}.
\]
\end{lemma}

The DA algorithms we study use BDF2 time stepping, and their analysis utilizes the $G$-norm and $G$-stability theory see e.g. \cite{HW02}, \cite{CGSW13}.  Briefly, define the matrix
\begin{align*}
G = \begin{bmatrix}
1/2 & -1 \\ -1 & 5/2
\end{bmatrix},  
\end{align*}
and note that $G$ induces the norm $\|x\|_G^2 := (x, Gx)$, which is equivalent to the $(L^2)^2$ norm:
\[ C_l \|x\|_G \leq \|x\| \leq C_u \|x\|_G  \]
where $C_l=3-2\sqrt{2}$ and $C_u=3+2\sqrt{2}$.
The following property is well-known \cite{HW02}. Set $\chi_v^n := [v^{n-1}, v^n]^T$, if $v^i \in L^2(\Omega)$, $i = n-1, n$, we have 
\begin{align}
\left(\frac{1}{2}(3v^{n+1} - 4v^n + v^{n-1}), v^{n+1}\right) = \frac{1}{2}(\|\chi_v^{n+1}\|_G^{2} - \|\chi_v^n\|_G^2) + \frac{1}{4} \|v^{n+1} - 2v^n + v^{n-1}\|^2 .
\label{Gidentity}
\end{align}

\section{A new interpolation operator for efficient continuous data assimilation}

Let $X_h=P_k(\tau_h)$ be a FE space consisting of globally continuous piecewise degree $k$ polynomials on a regular mesh $\tau_h$, and $X_H = P_0(\tau_H)$ be a FE space consisting of piecewise constant functions over a coarser mesh $\tau_H$.  We make the assumption that every element of $\tau_H$ contain at least one node from $\tau_h$, which is expected to be true since $\tau_H$ is typically much coarser than $\tau_h$ in practice.  For our purposes, one can assume that the nodes of $\tau_H$ are the points where observations of the true solution are made.

Denote by $\{ x_j\}_{j=1}^{M}$ the set of nodes of $X_h$, and $\{ x_{k_j} \}_{j=1}^N$ the set of nodes of $X_H$, noting that each coarse mesh node is also a fine mesh node.  Further, the coarse mesh nodes also satisfy the property that {\color{black} for each} coarse mesh element $E_j^H$, the node $x_{k_j}$ is contained in element $E_j^H$ and is closest to its center.  The assumed relationship between the fine and coarse meshes guarantees the existence of such a node for each element.

Denote the basis functions of these two FE spaces by
\begin{align*}
X_h : \{ \psi_1,\ \psi_2, \ \psi_3,\ ...,\ \psi_M \},\\
X_H : \{ \phi_1,\ \phi_2, \ \phi_3,\ ...,\ \psi_N \}.
\end{align*}
We assume the usual property of FE basis functions that $\psi_i(x_j)=1$ if $i=j$ and 0 otherwise, and $\phi_i(x_{k_j})=1$ if $j=i$ and 0 otherwise.  Note that this implies that $\phi_i=1$ on all of $E_i^H$ since basis functions of $X_H$ are piecewise constant.

To help define our new interpolant, we first consider the usual $L^2$ projection of a function $u\in L^2(\Omega)$ onto $X_H$, which is defined by: Find $P_{L^2}^H(u) \in X_H$ satisfying
\[
(P_{L^2}^H(u),v_H) = (u,v_H)\ \forall v_H\in X_H, 
\]
which is equivalent to
\[
(P_{L^2}^H(u),\phi_j) = (u,\phi_j)
\]
holding for all $1\le j\le N$.  Since $P_{L^2}^H(u) \in X_H$, we can write $P_{L^2}^H(u) = \sum_{m=1}^N \beta_m \phi_m$, and thus
\[
\sum_{m=1}^N \beta_m (\phi_m,\phi_j) = (u,\phi_j)
\]
for all $1\le j\le N$.  Since $X_H$ consists of piecewise constant basis functions with non-overlapping support, each of these equations reduces, yielding for $j=1,\ 2,\ ,...,\ N$,
\begin{equation}
 \beta_j (\phi_j,\phi_j) = (u,\phi_j) \implies \beta_j  = \frac{1}{area(E^H_j)}\int_{E^H_j} u\ dx. \label{L2proj1}
\end{equation}

We now define our new interpolation operator, denoted $\tilde{P}_{L^2}^H$, by
\begin{equation}
\tilde{P}_{L^2}^H(u) = \sum_{j=1}^N u(x_{k_j}) \phi_j. \label{newinterp}
\end{equation}
This operator can be considered an approximation of $P_{L^2}^H$, as it differs only in that the last integral in \eqref{L2proj1} is approximated with a quadrature rule that is exact on constants.  Indeed, if on each coarse mesh element $E_j^H$, we make the quadrature approximation in \eqref{L2proj1} by
\[
\int_{E^H_j} u\ dx \approx u(x_{k_j}) area(E^H_j),
\]
then the new interpolation operator $\tilde{P}_{L^2}^H$ is recovered from ${P}_{L^2}^H$.

\subsection{Implementation of the nudging term with interpolation operator $\tilde{P}_{L^2}^H$}

A key property of $\tilde{P}_{L^2}^H$ is how it acts on the basis functions of $X_h$.  For each basis function $\psi_i$ of $X_h$, we calculate using \eqref{newinterp} that
\[
\tilde{P}_{L^2}^H(\psi_i) = \sum_{j=1}^N \psi_i( x_{k_j} ) \phi_j = \left\{ \begin{array}{l l}  \phi_j &\mbox{ if } i=k_j, \\ 0 &\mbox{ else. }\end{array} \right. 
\]
Thus for each coarse mesh (piecewise constant) basis function, there is exactly one fine mesh basis function that $\tilde P_{L^2}^H$ maps to it (the $k_j^{th}$ basis function); all other fine mesh basis functions get mapped to zero by $\tilde P_{L^2}^H$.  Hence for the $M$ finite element basis functions $\psi_i$, $i=1,\ 2,\ ...,\ M$, the new operator $\tilde{P}_{L^2}^H$ maps $N$ of them one to one and onto the $X_H$ basis functions, and maps the other M-N of them to 0.

Consider now the FE implementation of the nudging term using this new interpolation operator $\tilde{P}_{L^2}^H$.  It will be written in DA algorithms as (see sections 3 and 4) \footnote{While it is typical for the weak formulation of DA nudging terms to take the form $\mu(I_H(u_h),v_h)$, applying the interpolation operator to the test function as well is necessary for a simple and efficient implementation, and as we show in later sections this does not adversely affect stability or convergence results of the associated DA algorithms.}
\[
\mu (\tilde P_{L^2}^H(u_h),\tilde P^H_{L^2}({\color{black}\chi_h})),
\]
and so creates a matrix contribution to the linear system of the form
\begin{align*}
\mu D_{mn} & = \mu (\tilde P_{L^2}^H(\psi_m),\tilde P_{L^2}^H(\psi_n)) \\
& = \mu \left( \sum_{i=1}^N \psi_m(x_{k_i})\phi_i, \sum_{j=1}^N \psi_n(x_{k_j}) \phi_j \right)\\
& = \mu \sum_{j=1}^N \left(  \psi_m(x_{k_j})\phi_j, \psi_n(x_{k_j}) \phi_j \right),
\end{align*}
with the last step holding since the $\phi_i's$ are non-overlapping piecewise constants.  But since $\psi_m(x_{k_j})$ is only nonzero if $m=k_j$, we have shown that
{\color{black}
\[
D_{mn} = \left \{  \begin{array}{l l} area(E^H_j) & \mbox{if } n=m=k_j, \\ 0 & \mbox{else.} \end{array} \right. 
\]}
This reveals that $D$ is diagonal, and is nonzero only at entries $(k_j,k_j)$.

The right hand side nudging term takes the form $\mu (\tilde{P}_{L^2}^H(u_{true}),\tilde{P}_{L^2}^H({\color{black}\chi_h}))$, and we can similarly derive
\[
 \mu (\tilde P_{L^2}^H(u_{true}),\tilde P_{L^2}^H(\psi_m)) 
 = \mu \left( \sum_{i=1}^N u_{true}(x_{k_i})\phi_i, \sum_{j=1}^N \psi_m(x_{k_j}) \phi_j \right)
 = \mu \sum_{j=1}^N \left(  u_{true}(x_{k_j})\phi_j, \psi_m(x_{k_j}) \phi_j \right) .
\]
Hence if $m\ne k_j$ for all $1\le j\le N$, then the term is zero,  but otherwise
\begin{equation}
 \mu (\tilde P_{L^2}^H(u_{true}),\tilde P_{L^2}^H(\psi_m)) = \mu\ u_{true}(x_{k_j}) \ area(E^H_j). \label{Dutrue}
\end{equation}
Denoting the vector $\hat u_{true}$ by $\hat u_{{true}_j} = u_{true}(x_j)$, we can write this right hand side nudging contribution as $\mu D \hat{u}_{true}$.

\begin{remark}
While it would potentially jeopardize the relevance of the convergence analysis in the following sections,
algebraic nudging could be implemented with 
\[
\tilde \mu D_{mn} =  \left \{  \begin{array}{l l}  \tilde \mu & \mbox{if } n=m=k_j, \\ 0 & \mbox{else.} \end{array} \right. 
\]
In this case, one could still consider $D$ to be the matrix arising from nudging, but with $\mu$ chosen locally to produce
$\tilde \mu$.  On quasi-uniform meshes, this could be a reasonable approach, if an even simpler implementation is desired.
\end{remark}

\subsection{Properties of $\tilde{P}_{L^2}^H$}

We now prove the fundamental stability and accuracy properties for the new interpolation operator.

\begin{lemma}\label{mainlemma}
For a given mesh $\tau_H(\Omega)$ with convex elements and maximum element diameter $H\le O(1)$, the operator $\tilde{P}_{L^2}^H$  satisfies for any $w\in H^1(\Omega)$,
\begin{align}
\| \tilde{P}_{L^2}^H  (w) - w \| &\le C H \| \nabla w\|, \label{interp1}
\\ \| \tilde{P}_{L^2}^H(w) \| &\leq  \|w \| + C H \| \nabla w \|. \label{interp2}
\end{align}

\end{lemma}
\begin{proof}
We begin our proof of \eqref{interp1} by expanding $\tilde{P}_{L^2}^H$ and ${P}_{L^2}^H$ in the $X_H$ basis, using their definitions:
\begin{align}
 {P}_{L^2}^H  (w) & = \sum_{i=1}^N \alpha_j \phi_j, \ \ \ \left(\alpha_j=\frac{1}{area(E^H_j)} \int_{E^H_j} w \ dx\right), \label{adef} \\
 \tilde{P}_{L^2}^H  (w) & = \sum_{i=1}^N \beta_j \phi_j, \ \ \ \left(\beta_j=w( x_{k_j}) \right). \label{bdef}
\end{align}
Now subtracting their difference inside of the $L^2(\Omega)$ norm yields, thanks to these basis functions being non-overlapping piecewise constants,
\begin{align}
\|  \tilde{P}_{L^2}^H  (w) -  {P}_{L^2}^H  (w) \| 
& = \norm{ \sum_{i=1}^N (\alpha_i - \beta_i) \phi_i } \nonumber  \\
& = \sum_{i=1}^N \bigg| \alpha_i - \beta_i \bigg| \norm{ \phi_i } \nonumber \\
& = \sum_{i=1}^N area(E^H_i) \cdot \bigg| \alpha_i - \beta_i \bigg|. \label{projprop1}
\end{align} 
Consider now $| \alpha_i - \beta_i |$.  Using the definitions of $\alpha_i$ and $\beta_i$, we bound the difference by
\begin{align*}
| \alpha_i - \beta_i |& = \bigg| \frac{1}{area(E^H_j)} \int_{E^H_j} w\ dx - w( x_{k_j} )  \bigg|
\end{align*}
This difference is precisely the error in a quadrature rule over $E^H_i$ that is exact on constants, and thus one can obtain the bound
\begin{align*}
| \alpha_i - \beta_i |& \le  \frac{diam(E^H_i)}{area(E^H_i)^{1/2}} \| \nabla w \|_{L^2(E^H_i)}.
\end{align*}
Combining this estimate with \eqref{projprop1} provides
\begin{align}
\|  \tilde{P}_{L^2}^H  (w) -  {P}_{L^2}^H  (w) \| 
& \le 
\sum_{i=1}^N area(E^H_i)^{1/2} diam(E^H_i) \| \nabla w \|_{L^2(E^H_i)} \nonumber \\
& \le \max_i \left( area(E^H_i)^{1/2} diam(E^H_i) \right) \sum_{i=1}^N \| \nabla w \|_{L^2(E^H_i)} \nonumber \\
& =  \max_i \left( area(E^H_i)^{1/2} diam(E^H_i) \right)  \| \nabla w \|_{L^2(\Omega)}. 
\end{align}
Now using the assumption that the maximum element diameter is $H$, we have 
\[
\|  \tilde{P}_{L^2}^H  (w) -  {P}_{L^2}^H  (w) \| \le C H^{1+ d/2} \| \nabla w \|.
\]
This result, the triangle inequality, and Proposition 1.135 in \cite{ern2013theory} now provide
\begin{align*}
\|  \tilde{P}_{L^2}^H  (w) -   w \| 
& \le  
\|  \tilde{P}_{L^2}^H  (w) -  {P}_{L^2}^H  (w) \|  + 
\|  {P}_{L^2}^H  (w) -   w \| \\
& \le 
C H^{1+ d/2} \| \nabla w \| + C H \| \nabla w \| \\
& \le 
C H \| \nabla w\|,
\end{align*}
which proves \eqref{interp1}.  The result \eqref{interp2} follows immediately from \eqref{interp1} and the triangle inequality.

\end{proof}

\section{Application: Data assimilation in fluid transport equations}

As a first application, we consider applying DA with the new interpolation operator to the fluid transport equation, given by
\begin{eqnarray}
c_t + U \cdot \nabla c - \epsilon \Delta c &= &f, \label{ft1} \\
c(0) & = & c_0, \label{ft2}
\end{eqnarray}
with boundary conditions $c |_{\Gamma_1}=0$ and $\nabla c \cdot n |_{\Gamma_2}=0$, where $\partial \Omega = \Gamma_1 \cup \Gamma_2$ and meas$(\Gamma_1\cap \Gamma_2)=0$.  The DA algorithm we consider is given as follows, with a regular, conforming finite element mesh $\tau_h$, function space
\[
X_h = \{v \in H^1(\Omega),\ v|_{\Gamma_1}=0\} \cap P_k(\tau_h),
\]
and appropriately chosen coarse mesh $\tau_H$ and $X_H = P_0(\tau_H)$ (constructed as discussed in section {\color{black} 2}).  For simplicity of analysis, we will assume a smooth boundary and that $\partial\Omega=\Gamma_1$, however in the numerical tests we do use mixed boundary conditions.

The DA algorithm we consider reads as follows, with a BDF2 temporal discretization and FE spatial discretization.

\begin{algorithm} \label{ftda}
	Given any initial conditions $c_h^0,\ c_h^{1} \in X_h$, divergence free velocity field $U\in L^{\infty}(\Omega)$, forcing $f \in L^\infty(0,\infty; L^2(\Omega))$, true solution $c \in L^\infty(0,\infty; L^2(\Omega))$,  and nudging parameter $\mu\ge 0$, find $c_h^{n+1}\in X_h$  for $n = 1,2,...$, satisfying for all $\chi_h\in X_h$,
	\begin{align}
	\frac{1}{2\Delta t} \left( 3c_h^{n+1} - 4c_h^n + c_h^{n-1},\chi_h \right) & + (U \cdot \nabla c_h^{n+1},\chi_h) \nonumber \\
	+ \epsilon (\nabla v_h^{n+1},\nabla \chi_h) & + \mu (\tilde{P}_{L^2}^H (c_h^{n+1} - c(t^{n+1})),\tilde{P}_{L^2}^H \chi_h)=  {\color{black}(f^{n+1},\chi_h).} \label{femft1} 
	\end{align}
\end{algorithm}

The implementation of Algorithm \ref{ftda} is rather straightforward.  Standard finite element packages can construct  the matrices $M$, $S$, and $N$ arising from $(c_h^{n+1},\chi_h)$, $(\nabla c_h^{n+1},\nabla {\color{black}\chi_h})$, and $(U\cdot\nabla c_h^{n+1},\chi_h)$, respectively.  Once the observation points $\{ x_{k_j} \}$ are defined, a coarse mesh can be constructed so that element $E_j^H$ contains $x_{k_j}$.  Exactly how to construct the coarse mesh is somewhat arbitrary, so long as the elements are convex, one can calculate the area of each element, and a reasonable minimum angle condition is enforced.  Once this is done, the diagonal nudging matrix $D$ can be constructed, as defined above.  This gives, at each time step, the linear algebraic system for the unknown coefficient vector
\[
\left( \frac{1.5}{\Delta t}M + N + \epsilon S + \mu D \right) \hat{c}^{n+1} = \frac{1}{\Delta t}M\left(2\hat c^n - \frac12 \hat c^{n-1}\right) + \hat{f}^{n+1} + \mu D \hat{c}_{true}^{n+1}.
\]
In this way, the method can easily be adapted to enable DA to work with existing and/or legacy codes.

\subsection{Analysis of the DA algorithm}

We prove in this section the long-time stability, well-posedness and accuracy of  Algorithm \ref{ftda}.  We begin with stability and well-posedness.  In all of our analysis, we invoke the $G$-norm and $G$-stability theory often used with BDF2 analysis, see e.g. \cite{HW02,CGSW13}.

\begin{lemma}
	For any $\dt>0$ and $\mu\ge 0$, Algorithm \ref{ftda} is well-posed globally in time, and solutions are long time stable: for any $n > 1$, 
	\begin{align*}
	\bigg({\color{black}C_u^{-2}} (\|c_h^{n+1}\|^2   + \|c_h^n\|^2) & + \frac{\epsilon\dt}{2}\|c_h^{n+1}\|^2   \bigg)
	\\ & \leq \bigg(C_u^2 (\|c_h^{1}\|^2   + \|c_h^0\|^2) + \frac{\epsilon\dt}{2}\|c_h^{1}\|^2 + \frac{\mu\dt}{2}\|c_h^{1}\|^2  \bigg) \left( \frac{1}{1+\lambda\dt}  \right)^{n+1} 
	\\ & \,\,\,\, + C\epsilon^{-1}\lambda^{-1}\|f\|_{L^\infty(0, \infty ; H^{-1})}^{2} + C\mu\lambda^{-1} \|\tilde{P}_{L^2}^H c(t^{n+1})\|^2,
	\end{align*} 
	where $\lambda = \min\{  2 \dt^{-1}, \frac{\epsilon C_P^{-2}C_l^{2}}{2}  \}$.
	\label{ftstab}
\end{lemma}
\begin{proof}
	Choose $\chi_h = c_h^{n+1}$, which vanishes the convective term, and, after dropping the non-negative term $\frac{1}{4\dt}\|c_h^{n+1}-2c_h^n+c_h^{n-1}\|^2$ on the left hand side, yields
	\begin{align}
	\frac{1}{2\dt}\|[c_h^{n+1}; c_h^n]\|^2_G& + \epsilon\|\nabla c_h^{n+1}\|^2 + \mu \|\tilde{P}_{L^2}^H c_h^{n+1}\|^2  \nonumber
	\\ & \leq \frac{1}{2\dt}\|[c_h^{n}; c_h^{n-1}]\|^2_G + \mu|(\tilde{P}_{L^2}^H c(t^{n+1}), \tilde{P}_{L^2}^H c_h^{n+1})| + |(f^{n+1}, c_h^{n+1})|. \label{stab1}
	\end{align}
	The nudging term on the right hand side is bounded using Cauchy-Schwarz and Young's inequalities to obtain
	\begin{align*}
	\mu|(\tilde{P}_{L^2}^H c(t^{n+1}), \tilde{P}_{L^2}^H c_h^{n+1})| 
	& \leq C \mu\|\tilde{P}_{L^2}^H c(t^{n+1})\|^2 + \frac{\mu}{2} \| \tilde{P}_{L^2}^H c_h^{n+1}\|^2.
	\end{align*}
	The forcing term is bounded using the $H^{-1}(\Omega)$ norm as well as Young's inequality, via
	\begin{align*}
	|(f^{n+1}, c_h^{n+1})| \leq \frac{\epsilon^{-1}}{2}\|f\|_{L^\infty(0,\infty ; H^{-1})}^2 + \frac{\epsilon}{2}\|\nabla c_h^{n+1}\|^2 .
	\end{align*}
	Replacing the right hand side of \eqref{stab1} with these bounds and multiplying by $2\dt$, we obtain 
	\begin{align*}
	\|[c_h^{n+1}; c_h^n]\|^2_G + & \epsilon \dt \|\nabla c_h^{n+1}\|^2 + \mu\dt \|\tilde{P}_{L^2}^H c_h^{n+1}\|^2  \nonumber
	\\ & \leq \|[c_h^{n}; c_h^{n-1}]\|^2_G + C\mu\dt \|\tilde{P}_{L^2}^H c(t^{n+1})\|^2 + \epsilon^{-1}\dt\|f\|_{L^\infty(0,\infty;H^{-1})}^2. 
	\end{align*}
	Next, drop the positive nudging term on the left hand side and add $\frac{\epsilon\dt}{4}\|\nabla c_h^{n}\|^2$ to both sides of the equation to obtain
	\begin{align*}
	\bigg(\|[c_h^{n+1}; c_h^n]\|^2_G + \frac{\epsilon\dt}{4} & \|\nabla c_h^{n+1}\|^2\bigg) +  \frac{\epsilon\dt}{4}\bigg(\|\nabla c_h^{n+1}\|^2 + \|\nabla c_h^n\|^2 \bigg) + \frac{\epsilon\dt}{2}\|\nabla c_h^{n+1}\|^2   \nonumber
	\\ & \leq \|[c_h^{n}; c_h^{n-1}]\|^2_G+ \frac{\epsilon\dt}{4}\|\nabla c_h^n\|^2 + C\mu\dt \|\tilde{P}_{L^2}^H c(t^{n+1})\|^2+ \epsilon^{-1}\dt\|f\|_{L^\infty(0,\infty;H^{-1})}^2, 
	\end{align*}
	which then reduces using $G$-norm equivalence and Poincar\'e's inequality,
	\begin{align}
	\bigg(\|[c_h^{n+1}; & c_h^n]\|^2_G  + \frac{\epsilon\dt}{4}  \|\nabla c_h^{n+1}\|^2\bigg) +  \frac{\epsilon C_P^{-2}C_l^2\dt}{4}\bigg(\| [c_h^{n+1};c_h^n]\|_G^2 \bigg) + \frac{\epsilon\dt}{2}\|\nabla c_h^{n+1}\|^2   \nonumber
	\\ & \leq \|[c_h^{n}; c_h^{n-1}]\|^2_G+ \frac{\epsilon\dt}{4}\|\nabla c_h^n\|^2 + C\mu\dt \|\tilde{P}_{L^2}^H c(t^{n+1})\|^2+ \epsilon^{-1}\dt\|f\|_{L^\infty(0,\infty;H^{-1})}^2. \label{sstab4}
	\end{align} 
	Using $\lambda = \min\left\{ 2\dt^{-1}, \frac{\epsilon C_P^{-2}C_l^2}{4}  \right\}$, equation \eqref{sstab4} can be written as
	\begin{align}
	(1+\lambda&\dt) \bigg(\|[c_h^{n+1}; c_h^n]\|^2_G  + \frac{\epsilon\dt}{4}  \|\nabla c_h^{n+1}\|^2\bigg)  \nonumber
	\\ &\leq \bigg(\|[c_h^{n}; c_h^{n-1}]\|^2_G+ \frac{\epsilon\dt}{4}\|\nabla c_h^n\|^2\bigg) + \dt (C\mu \|\tilde{P}_{L^2}^H c(t^{n+1})\|^2+ \epsilon^{-1}\|f\|_{L^\infty(0,\infty;H^{-1})}^2). \label{stab5}
	\end{align}
	{\color{black}Lastly, we use Lemma \ref{geoseries} to write}
	\begin{align*}
	&\|[c_h^{n+1}; c_h^n]\|^2_G  + \frac{\epsilon\dt}{4}  \|\nabla c_h^{n+1}\|^2 
	\\ &\leq \left(\frac{1}{1+\lambda\dt}\right)^{n+1}\bigg(\|[c_h^{1}; c_h^{0}]\|^2_G+ \frac{\epsilon\dt}{4}\|\nabla c_h^0\|^2\bigg) + \lambda^{-1} (C\mu \|\tilde{P}_{L^2}^H c(t^{n+1})\|^2+ \epsilon^{-1}\|f\|_{L^\infty(0,\infty;H^{-1})}^2). 
	\end{align*}
	The stability result is completed using the $G$-norm equivalence. At each time step, the scheme is linear and finite dimensional, and thus this stability result immediately implies existence and uniqueness of the solutions, and thus well-posedness of the algorithm. 
	
\end{proof}

Next, we prove that solutions to Algorithm \ref{ftda} converge to the true solution, exponentially fast in time, up to optimal discretization error, provided restrictions on $H$ and $\mu$.  Our analysis will use the $H^1_0$ projection onto $X_h$, denoted by $\pi_h$ and defined by: Given $\phi\in H^1(\Omega)$, $\pi_h \phi \in X_h$ satisfies
	\[
	(\nabla \pi_h \phi,\nabla v_h) = (\nabla \phi_h,\nabla v_h)
	\]
	for all $v_h \in X_h$. For $\phi\in H^1_0(\Omega)$, we have the following estimate \cite{J81},
	\[
	\| \pi_h \phi - \phi \| + h \| \nabla (\pi_h\phi - \phi) \| \le C h^{k+1} | \phi |_{k+1}.
	\]

\begin{theorem}
	Let $c\in L^\infty(0,\infty; H^{k+1})$ denote the true solution to the fluid transport equation with given $f\in L^\infty(0,\infty; L^2)$, initial condition $c^0\in L^2(\Omega)$, and $c_t, c_{tt}, c_{ttt} \in L^\infty(0, \infty; L^2)$. 
	Then for any $H$ and any $\mu\ge 0$ and $\Delta t>0$, the difference between the DA solution and the true solution satisfies, for all $n$, 
	\begin{align*}
		\|c^n - c_h^n\|^2 \leq \left( \frac{1}{1 + \lambda \dt}   \right)^n \|c_0 - c_h^0\|^2 + \frac{R}{\lambda }, 
	\end{align*}
	where $R = C \epsilon^{-1}h^{2k+2} + C\epsilon^{-1}\dt^4 + C\mu(h^{2k+2}+  H^{2}h^{2k}) + C \mu^{-1} h^{2k+2}$ and $\lambda = \min\left\{  2 \dt^{-1}, \frac{\epsilon C_P^{-2}{\color{black}C_l^{2}}}{4}   \right\}$.
	
	Furthermore, if $\mu \leq \frac{\epsilon}{CH^2}$, then we can take $\lambda = \min\left\{  2 \dt^{-1}, \frac{((\epsilon - C\mu H^2) C_P^{-2} + \mu ){\color{black}C_l^{2}}}{4}   \right\}$.
	\label{ftconv}
\end{theorem}

\begin{remark}
It is no surprise with fluid transport that the DA algorithm will converge to the true solution (up to discretization error), even when $\mu=0$.  This is because the initial condition will eventually diffuse away, leaving the forcing (and boundary conditions) to drive the system.  Hence if the algorithm has the correct forcing and boundary conditions, then it must converge to the true solution even without nudging, as we prove below.  However, if $H$ and $\mu$ satisfy the restriction, we have proven that the DA nudging can significantly speed up the convergence to the true solution, and we observe exactly this phenomena in our numerical tests.
\end{remark}

\begin{proof}
	After applying Taylor's theorem, the true solution satisfies, for all $\chi_h \in X_h$,
	\begin{align}
		\frac{1}{2\dt} (3c^{n+1} - 4c^n + c^{n-1}, \chi_h) + (U \cdot \nabla c^{n+1}, \chi_h) + &\epsilon(\nabla c^{n+1}, \nabla \chi_h)  \nonumber 
		\\ &= (f^{n+1}, \chi_h) + \frac{\dt^2}{3}(c_{ttt}(t^*), \chi_h), \label{transporttrue22}
	\end{align}
	where $t^* \in [t^{n-1}, t^{n+1}]$. Subtracting \eqref{ftda} and \eqref{transporttrue22} and letting $e^n = c^n - c_h^n$ yields the difference equation
	\begin{align}
	\frac{1}{2\dt} (3e^{n+1} - 4e^n + e^{n-1}, \chi_h)  + \epsilon(\nabla e^{n+1}, &  \nabla \chi_h) + \mu (\tilde{P}_{L^2}^He^{n+1}, \tilde{P}_{L^2}^H \chi_h)     \nonumber
	\\ & = - (U \cdot \nabla e^{n+1}, \chi_h)- \frac{\dt^2}{3}(c_{ttt}(t^*), \chi_h). \label{ftdiff}
	\end{align}
	 Now decompose the error by adding and subtracting $\pi_h(c^{n})$ to $e^n$, denote $\eta^n = c^n - \pi_h(c^{n})$ and $\phi_h^n = \pi_h(c^{n}) - c_h^n$ so that $e^n = \eta^n + \phi_h^n$, with $\phi_h^n \in X_h$. Choose $\chi_h = \phi_h^{n+1}$ and use the $G$-stability framework to write the difference equation as
	\begin{align}
	\frac{1}{2\dt} \bigg(  \|[\phi_h^{n+1};& \phi_h^n]\|_G^2 - \|[\phi_h^n; \phi_h^{n-1}]\|_G^2 \bigg)  + \epsilon\|\nabla \phi_h^{n+1}\|^2 + \mu \|\tilde{P}_{L^2}^H \phi_h^{n+1}\|^2 \nonumber
	\\ & \leq |(U \cdot \nabla \eta^{n+1}, \phi_h^{n+1})| + \frac{\dt^2}{3}|(c_{ttt}(t^*), \phi_h^{n+1})|   + \mu |(\tilde{P}_{L^2}^H\eta^{n+1}, \tilde{P}_{L^2}^H\phi_h^{n+1})| \nonumber 
	\\ & \,\,\,\, +  \left| \left( \frac{\eta^{n+1} - 4\eta^n + \eta^{n-1}}{2\dt}, \phi_h^{n+1}  \right)\right|.
	\label{ftdiff2}
	\end{align}
	Except for the nudging term, all right hand side terms above are bounded using standard inequalities, as in \cite{GR86, laytonbook,T77,erv031}: 
	\begin{align*}
	|(U \cdot \nabla \eta^{n+1}, \phi_h^{n+1})| & \leq C \epsilon^{-1}  \| \eta^{n+1}\|^2 + \frac{\epsilon}{8}  \|\nabla \phi_h^{n+1}\|^2 , 
	\\ \frac{\dt^2}{3}|(c_{ttt}(t^*), \phi_h^{n+1})| 
	& \leq C\epsilon^{-1} \dt^4 \|c_{ttt}\|^2_{L^\infty(0,\infty;L^2)} + \frac{\epsilon}{8}\|\nabla \phi_h^{n+1}\|^2,
	\\ \frac{1}{2\dt} |(3\eta^{n+1} - 4 \eta^n + \eta^{n-1}, \phi_h^{n+1})| & \leq C\epsilon^{-1}\left(\|\eta_t\|^2_{L^\infty(0,\infty , L^2)} + \int_{t^{n-1}}^{t^{n+1}} \|\eta_{tt}\|^2 dt\right)
	+ \frac{\epsilon}{4}\|\nabla \phi_h^{n+1}\|^2.
	\end{align*}
	For the nudging term on the right hand side, we first apply Cauchy-Schwarz and Young's inequalities, and then inequality \eqref{interp2}, which yields
	\begin{align*}
	\mu |(\tilde{P}_{L^2}^H\eta^{n+1}, \tilde{P}_{L^2}^H\phi_h^{n+1})| & \leq \mu \|\tilde{P}_{L^2}^H\eta^{n+1} \|^2 + \frac{\mu}{4}\|\tilde{P}_{L^2}^H\phi_h^{n+1}\|^2
	\\ & \leq C\mu (\|\eta^{n+1}\|^2 + H^2\|\nabla \eta^{n+1}\|^2 ) + \frac{\mu}{4}\|\tilde{P}_{L^2}^H\phi_h^{n+1}\|^2.
	\end{align*}
	Using these bounds in \eqref{ftdiff2}, dropping the nudging term on the left hand side, and multiplying by $2\dt$ gives
		\begin{align}
	  \|[{\color{black}\phi_h^{n+1}}&; \phi_h^n]\|_G^2  + \epsilon\dt \|\nabla \phi_h^{n+1}\|^2  \nonumber
	 \\ & \leq \|[\phi_h^n; \phi_h^{n-1}]\|_G^2 + C \epsilon^{-1} \dt \| \eta^{n+1}\|^2 + C\epsilon^{-1} \dt^5 \|c_{ttt}\|^2_{L^\infty(0,\infty;L^2)} \nonumber 
	 \\ & \,\,\,\,
	  + C\mu \dt (\|\eta^{n+1}\|^2 + H^2\|\nabla \eta^{n+1}\|^2 ) + C \epsilon^{-1} \dt \left(  \|\eta_t\|^2_{L^\infty(0,\infty; L^2)} + \int_{t^{n-1}}^{t^{n+1}} \|\eta_{tt}\|^2 dt  \right) \nonumber 
	  \\ & \leq  \|[\phi_h^n; \phi_h^{n-1}]\|_G^2 + C \dt \epsilon^{-1}h^{2k + 2} + C\dt^5 + C \mu \dt(h^{2k+2} + H^2h^{2k}) + C\epsilon^{-1}\dt h^{2k+2}.
	\label{ftdiff3}
	\end{align}
	Let $R := C  \epsilon^{-1}h^{2k + 2} + C\dt^4 + C \mu (h^{2k+2} + H^2h^{2k}) + C\epsilon^{-1} h^{2k+2}$ and after applying approximation properties of $\pi_h$, apply Poincar\'e's inequality on the left hand side to reveal
	\begin{align*}
		 \|[\phi_h^{n+1}; \phi_h^n]\|_G^2  + \epsilon C_P^{-2}\dt \|\phi_h^{n+1}\|^2  \leq \|[\phi_h^n; \phi_h^{n-1}]\|_G^2 + \dt R.
	\end{align*}
	We can now proceed as in the long time stability proof above to get 
	\begin{align*}
	\|[\phi_h^{n+1}; \phi_h^n]\|_G^2  + \epsilon\dt \|\phi_h^{n+1}\|^2  \leq \|[\phi_h^1; \phi_h^{0}]\|_G^2 \left(\frac{1}{1 + \lambda \dt }\right)^{n+1} + \frac{R}{\lambda}.
	\end{align*}
	$G$-norm equivalence and triangle inequality complete the first part of proof, without any restriction on $H$ or $\mu$.
	
We will now show that with an added restriction on $\mu$, we obtain a faster convergence rate to the true solution. Starting back at \eqref{ftdiff}, add and subtract $\phi_h^{n+1}$ in both components of the nudging inner product on the left hand side to obtain 
		\begin{align}
	\frac{1}{2\dt} \bigg(  \|[{\color{black}\phi_h^{n+1}}; \phi_h^n]\|_G^2& - \|[\phi_h^n; \phi_h^{n-1}]\|_G^2 \bigg)  + \epsilon\|\nabla \phi_h^{n+1}\|^2 + \mu \|\phi_h^{n+1}\|^2 \nonumber
	\\ & \leq |(U \cdot \nabla \eta^{n+1}, \phi_h^{n+1})| + \frac{\dt^2}{3}|(c_{ttt}(t^*), \phi_h^{n+1})|   \nonumber 
	+ \mu |(\tilde{P}_{L^2}^H\eta^{n+1}, \tilde{P}_{L^2}^H\phi_h^{n+1})| 
	\\ & \,\,\,\,  + \mu \|\tilde{P}_{L^2}^H \phi_h^{n+1} - \phi_h^{n+1}\|^2 + 2 \mu |(\tilde{P}_{L^2}^H \phi_h^{n+1} - \phi_h^{n+1}, \phi_h^{n+1})|    \nonumber
	\\ & \,\,\,\, + \frac{1}{2\dt} |(3\eta^{n+1} - 4 \eta^n + \eta^{n-1}, \phi_h^{n+1})|.
	\label{aftdiff2}
	\end{align}
	This then leads to two additional right hand side terms to bound, and we have to adjust the bound on the original right hand side nudging term by applying inequality \eqref{interp2}.
	We upper bound the first one with inequality \eqref{interp1}, yielding
	\begin{align*}
	\mu \|\tilde{P}_{L^2}^H \phi_h^{n+1} - \phi_h^{n+1}\|^2 \leq C\mu H^2 \| \nabla \phi_h^{n+1}\|^2. 
	\end{align*}
	In a similar manner, we start with Cauchy-Schwarz on the last nudging term, then apply inequality \eqref{interp1} and Poincar\'e's inequality to obtain 
	\begin{align*}
	2 \mu |(\tilde{P}_{L^2}^H \phi_h^{n+1} - \phi_h^{n+1}, \phi_h^{n+1})| & \leq C \mu \|\tilde{P}_{L^2}^H \phi_h^{n+1} - \phi_h^{n+1}\| \|\phi_h^{n+1}\|
	\\ & \leq C\mu H^2 \| \nabla \phi_h^{n+1}\|^2 + \frac{\mu}{4}\|\phi_h^{n+1}\|^2
	\end{align*}
	Replace the right hand side of \eqref{aftdiff2} with the bounds above, reduce, and multiply by $2\dt$ to obtain
	\begin{align*}
	\|[{\color{black}\phi_h^{n+1}}; & \phi_h^n]\|_G^2  + \dt (\epsilon - C\mu H^2) \|\nabla \phi_h^{n+1}\|^2 + \mu\dt \|\phi_h^{n+1}\|^2 \nonumber
	\\ & \leq \|[\phi_h^n; \phi_h^{n-1}]\|_G^2 + C \epsilon^{-1} \dt \| \eta^{n+1}\|^2 + C\epsilon^{-1} \dt^5 \|c_{ttt}\|^2_{L^\infty(0,\infty;L^2)} \nonumber 
	\\ & \,\,\,\, + C\mu \dt (\|\eta^{n+1}\|^2 + H^2 \|\nabla \eta^{n+1}\|^2 ) + C\mu^{-1} \dt \|\eta_t\|^2_{L^\infty(0,\infty , L^2)} + C \mu^{-1} \dt \int_{t^{n-1}}^{t^{n+1}} \|\eta_{tt}\|^2 dt\nonumber 
	\\ & \leq  \|[\phi_h^n; \phi_h^{n-1}]\|_G^2 + C \dt \epsilon^{-1}h^{2k+2} + C\epsilon^{-1}\dt^5 + C\mu\dt(h^{2k+2}+  H^{2}h^{2k}) + C \mu^{-1}\dt h^{2k+2} .
	\end{align*}
	This then leads to the restriction on $\mu$, 
	$$
	\epsilon - C\mu H^2 \geq 0 \,\, \iff \,\, \mu \leq \frac{\epsilon}{CH^2}.
	$$
	With $R := C  \epsilon^{-1}h^{2k+2} + C\epsilon^{-1}\dt^4 + C\mu(h^{2k+2}+  H^{2}h^{2k}) + C \mu^{-1}h^{2k+2} $, we can complete the proof using equivalent arguments as above. 
\end{proof}

\subsection{Numerical Tests}

We now give results for numerical tests of Algorithm \ref{ftda}.  We illustrate the predicted convergence rates with respect to the discretization parameters $h$, $H$ and $\Delta t$, the convergence in time for changing $\mu$ and $H$, and also show the method works very well on a more practical test problem.

\subsubsection{Convergence rates}

\begin{figure}[!ht]
\begin{center}
$\tau_h$ (h=1/16) \hspace{1.6in}  $\tau_H$ (H=1/4) \\
\includegraphics[width = .35\textwidth, height=.3\textwidth,viewport=0 0 550 400, clip]{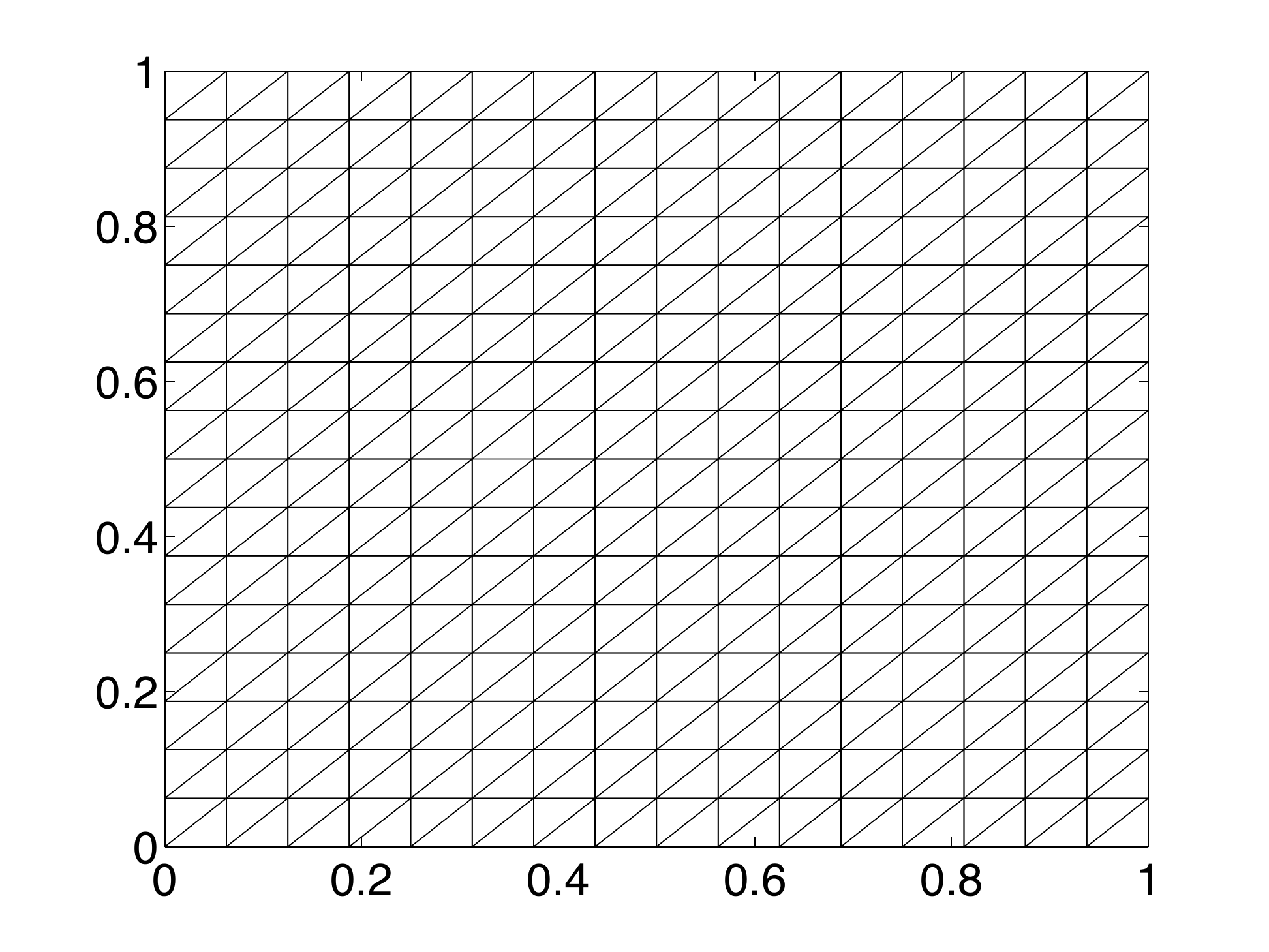} \ \ \ \ \ \ \ \
\includegraphics[width = .35\textwidth, height=.3\textwidth,viewport=0 0 550 400, clip]{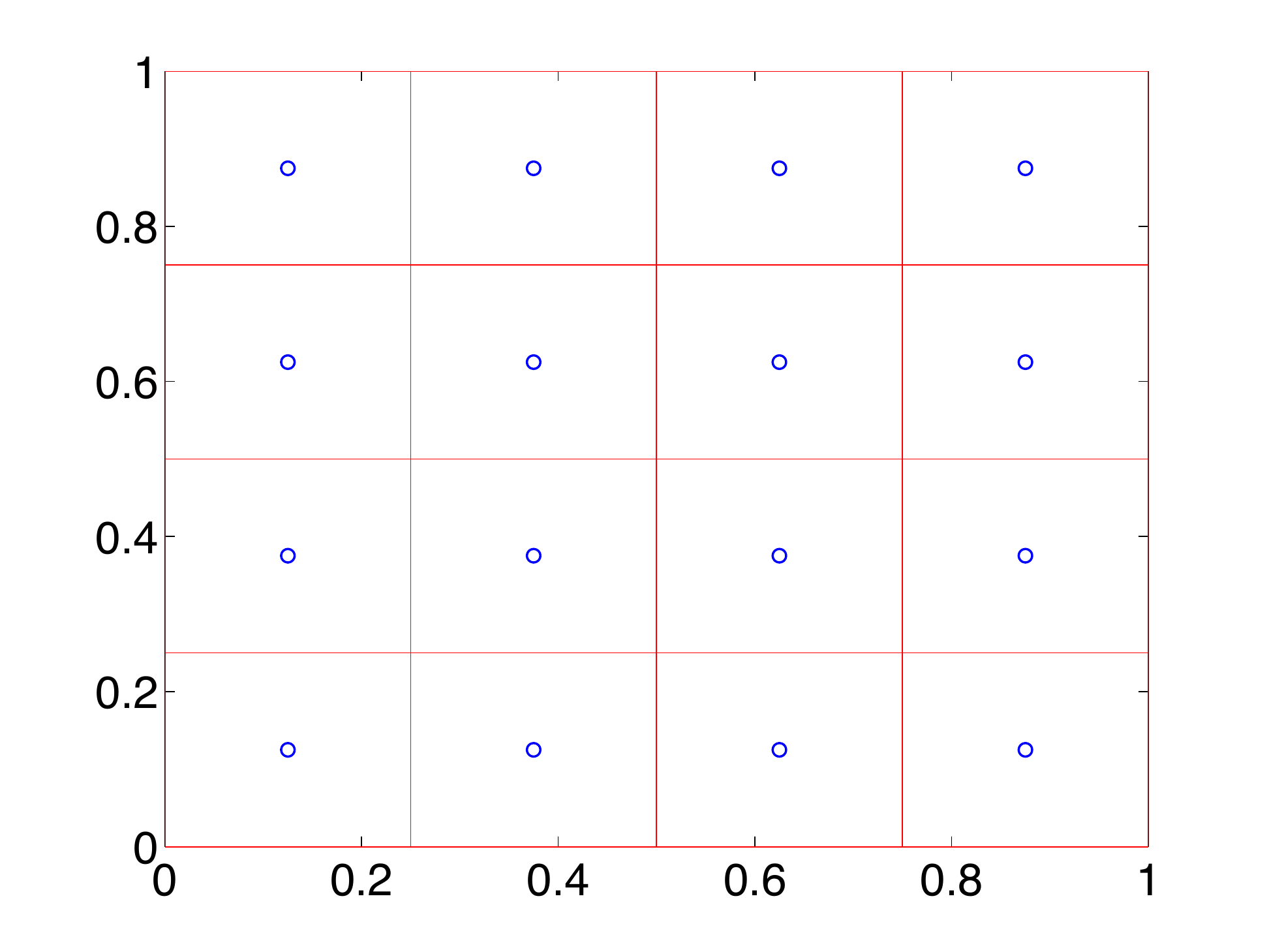}
	\caption{\label{meshes1} Shown above is an example of a fine mesh and associated coarse mesh used in the convergence rate test.}
	\end{center}
\end{figure}

To test the convergence rates with respect to the discretization parameters, we 
select the true solution $c=\sin(x+y+t)$ on $\Omega=(0,1)^2 \times [0,5]$, with transport 
velocity $U=\langle 1,0 \rangle^T$, and $\epsilon=1$.  The forcing $f$ is calculated from this 
chosen solution and \eqref{ft1}.  We compute approximate solutions using Algorithm \ref{ftda} with the calculated $f$, $P_2$ finite elements, Dirichlet boundary conditions enforced nodally to be equal to the true solution, zero initial conditions $c_h^0=c_h^1=0$, $\mu=1$, a uniform triangular mesh $\tau_h$, and a uniform square mesh $\tau_H$, see figure \ref{meshes1}.  Computations are done with varying $h$, $H$, and $\Delta t$, but we tie the discretization parameters together via $h=H/4$ and $\Delta t = Ch^{3/2}$ with $C$=0.9051.  We then successively refine $h$ (and thus $H$ and $\Delta t$ as well), and calculate the $L^2$ norm of the difference to the true solution at $t$=5.  Due to the way the parameters are tied together, we expect from our above theory that $\| c_h^n - c(t^n) \|_{L^2} = O(h^3)$, which is exactly what we observe in table \ref{hconv}.

\begin{table}[h!]
\centering
\begin{tabular}{|c | c | c |}
\hline
$h$ & $\| c_h - c \|_{L^2}$ & Rate \\ \hline
1/8 &    9.1235e-05     &       - \\ \hline
 1/16 &   1.1249e-05  &  3.02 \\ \hline
 1/32 &   1.4136e-06  &  2.99 \\ \hline
 1/64 &   1.7687e-07  &  3.00\\ \hline
\end{tabular}
\caption{\label{hconv} Convergence rates of Algorithm \ref{ftda} to the true solution at $t$=5, for varying $h$, $H=4h$ and $\Delta t = 0.9051 h^{3/2}$.  Third order convergence is observed, as predicted by the theory.}
\end{table}

\subsubsection{Effect of $\mu$ and $H$ on convergence to the true solution as $t\rightarrow \infty$}

\begin{figure}[!ht]
\begin{center}
H=1/4 \hspace{1.5in} H=1/8 \hspace{1.5in} H=1/32 \\
\includegraphics[width = .32\textwidth, height=.28\textwidth,viewport=0 0 530 400, clip]{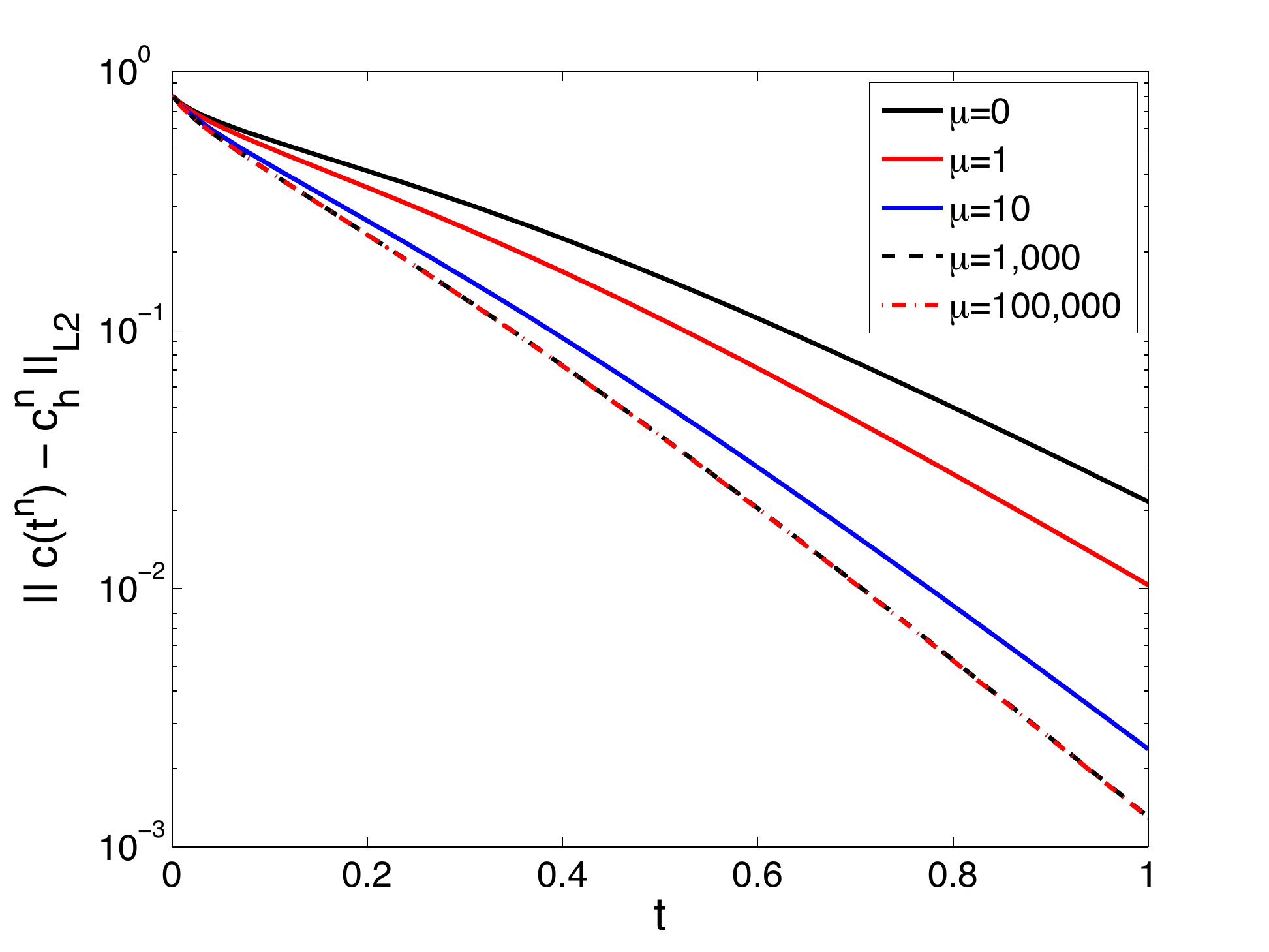} 
\includegraphics[width = .32\textwidth, height=.28\textwidth,viewport=0 0 530 400, clip]{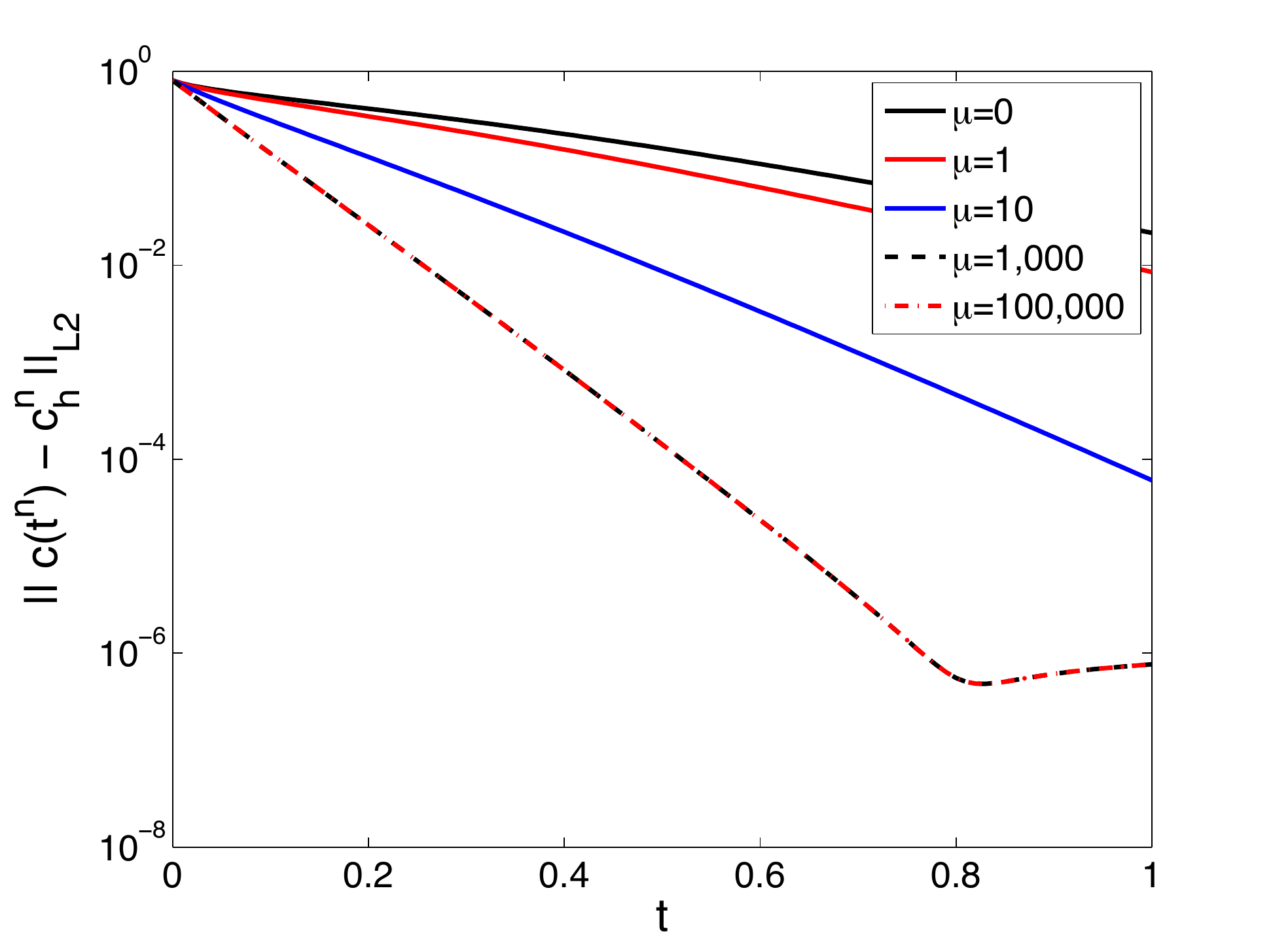}
\includegraphics[width = .32\textwidth, height=.28\textwidth,viewport=0 0 530 400, clip]{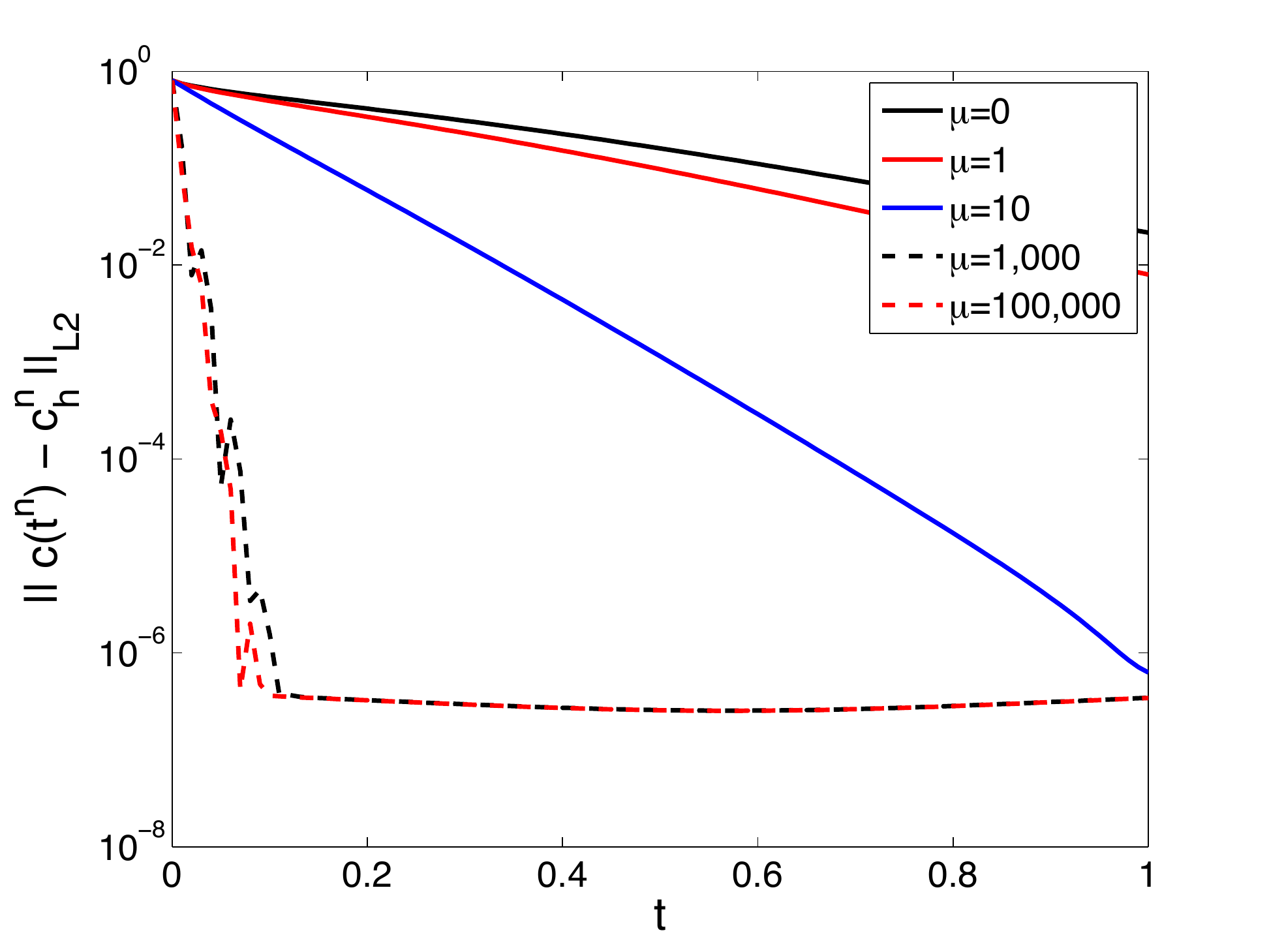} 
	\caption{\label{muvary} Shown above is the $L^2$ difference to the true solution versus time, for DA simulations of fluid transport  with varying $\mu$ and $H$.}
	\end{center}
\end{figure}

Our theory predicts that even with $\mu=0$, the DA solution will converge to the true solution (up to discretization error), exponentially fast in time.  However, we also prove that under restrictions that $\mu H^2$ is sufficiently small, the {\color{black} speed of} convergence will be increased as $\mu$ increases (until it becomes so large that $\mu H^2$ is no longer sufficiently small).  We test this theory now.  

Repeating the test above for convergence rates, but now with $\epsilon=0.01$, the fine mesh fixed to be a uniform triangular mesh with $h=1/32$, and time step size $\Delta t=0.01$ fixed.  We then run Algorithm \ref{ftda} with varying $H$ and $\mu$, calculating for each run the $L^2$ error versus time.
Results of these tests are shown in  figure \ref{muvary}.  We observe that for the largest $H=1/4$, convergence is only slightly increased with $\mu=1$ over $\mu=0$, and even very large $\mu$ does not produce significant speed up in the convergence.  However, as $H$ is decreased, we observe that large $\mu$ has a much greater impact; in particular, for $H=$1/32, and $\mu\ge 1,000$, convergence to the true solution is almost immediate.

\subsubsection{DA prediction of contaminant transport}

\begin{figure}[!ht]
\begin{center}
Fine mesh $\tau_h$\\
\includegraphics[width = .7\textwidth, height=.15\textwidth,viewport=50 0 650 300, clip]{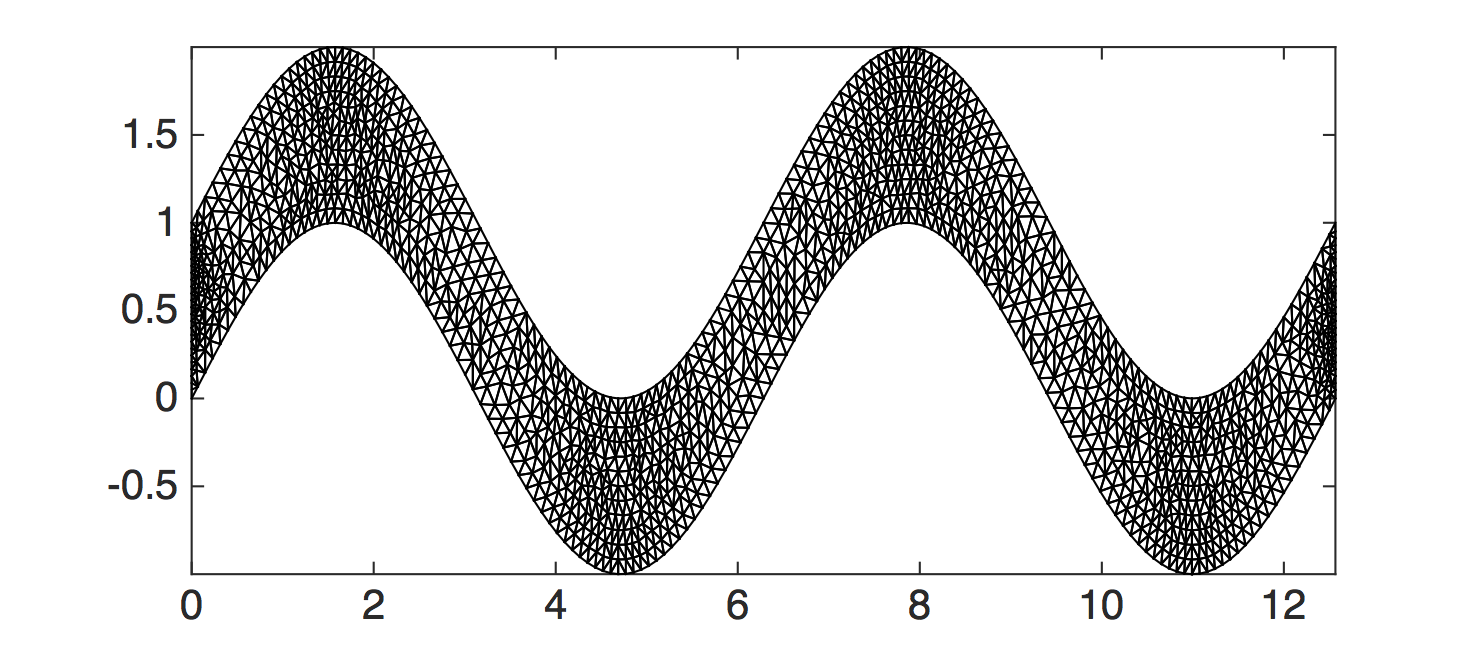}\\
Coarse mesh $\tau_H$ is the intersection of $\Omega$ and the rectangular mesh \\ \ \ \ \ \ \ \
\includegraphics[width = .65\textwidth, height=.15\textwidth,viewport=90 30 650 300, clip]{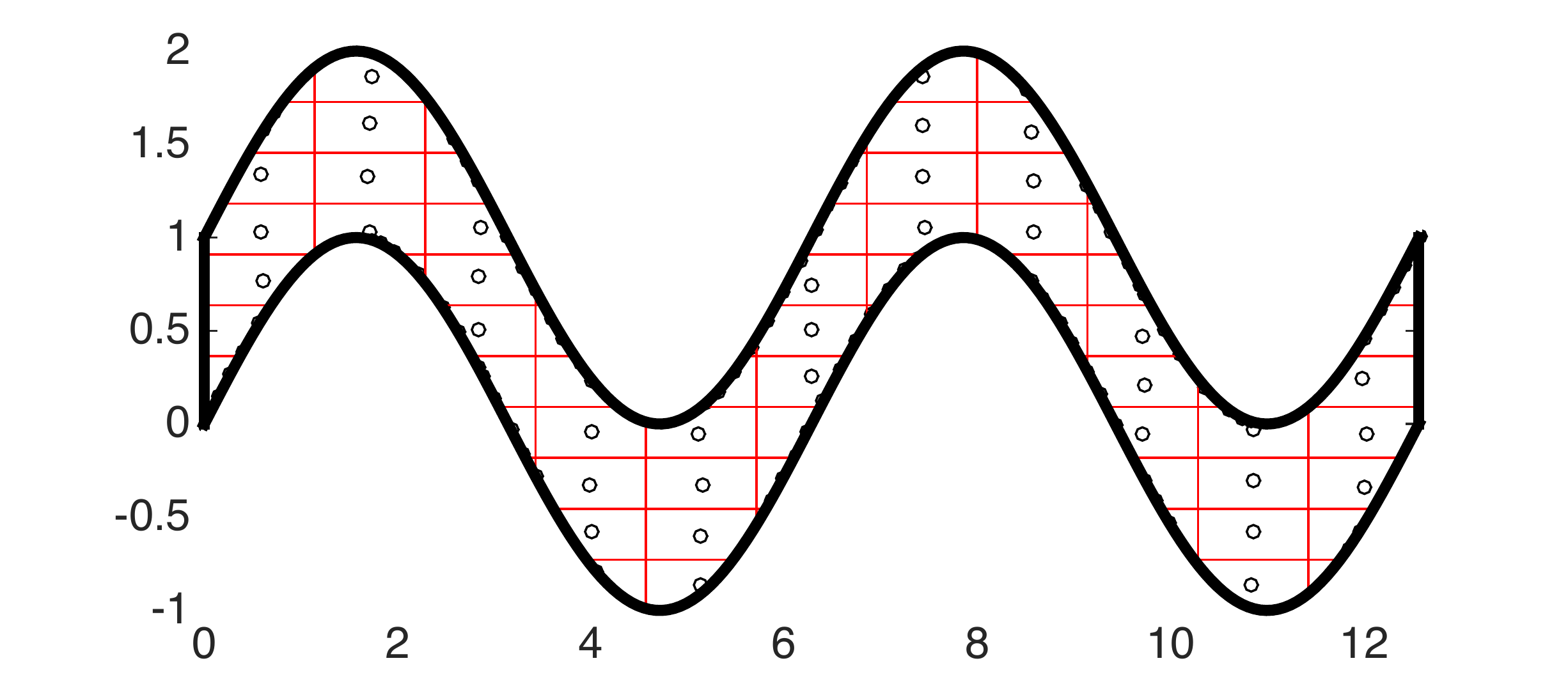}
	\caption{\label{meshes} Shown above are the fine and coarse meshes used in the fluid transport numerical
	test.  The coarse mesh nodes $x_{k_j}$ are shown as black circles.}
	\end{center}
\end{figure}

Our last test for Algorithm \ref{ftda} is on the following test problem, which is intended to simulate
contaminant transport in a river.  The domain is constructed from the curves $y=\sin(x)$ and $y=1+\sin(x)$ as lower and
upper boundaries, with $x=0$ and $x=4\pi$ as left and right boundaries.  Using the mesh $\tau_h$ shown in figure
\ref{meshes}, we use $((P_2)^2,P_1)$ Taylor-Hood elements (which gives a total of 13,928 total degrees of freedom (dof)) to compute a solution to the Stokes equations with viscosity $0.01$ and zero forcing, using no-slip boundary conditions on the top and bottom boundaries, a plug inflow of $u_{in}=3$, and a zero-traction outflow enforced with the do-nothing condition (see e.g. \cite{laytonbook} for more details on FE implementation of Stokes equations).  We take our transport velocity $U$ to be the velocity solution of this discrete Stokes problem.  

\begin{figure}[!ht]
\begin{center}
\includegraphics[width = .45\textwidth, height=.28\textwidth,viewport=0 0 650 360, clip]{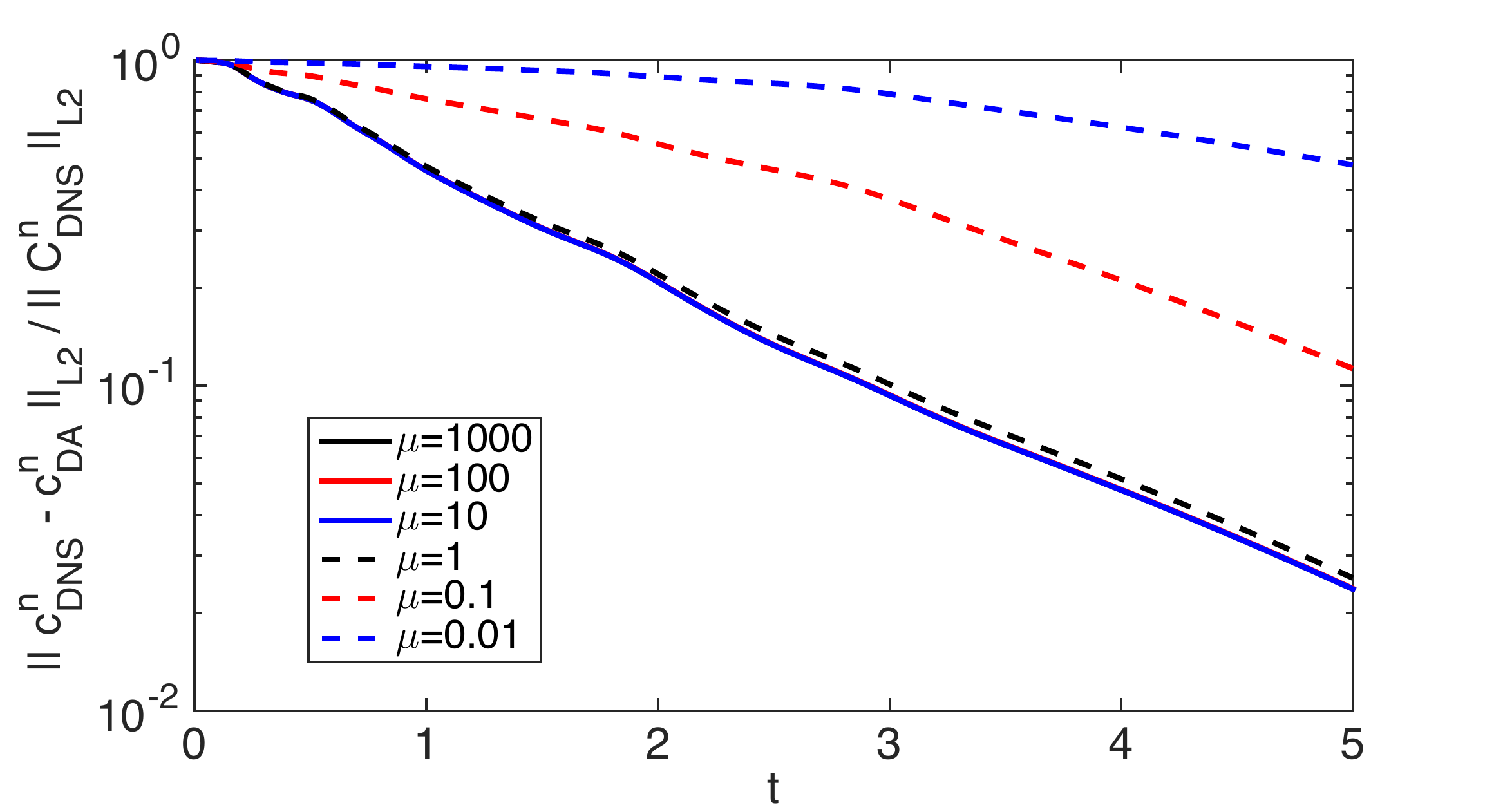}\\
	\caption{\label{conv} Shown above are the fine and coarse meshes used in the fluid transport numerical
	test.  The coarse mesh is created by intersecting the rectangular grid with the domain, and for each coarse mesh
	element $E^H_j$, we also plot the fine mesh node $x_{k_j}$ that is closest to the center of $E_j^H$.}
	\end{center}
\end{figure}

We next solve \eqref{ft1}-\eqref{ft2} equipped with a boundary condition of 0 contaminant at the inflow ($c_{in}=0$), the transport velocity $U$ from the discrete Stokes problem above, $\epsilon=0.01$, and 
zero Neumann conditions  $\nabla c \cdot n =0$ at all other boundaries.  For the initial condition, there is zero contaminant except for two `blobs', which are represented with $c=3$ inside the circles centered at $(1,1.5)$ and $(5,-0.5)$, with radius $0.1$ (see figure \ref{dacontour}, top right plot).  We compute a direct numerical simulation (DNS) for the concentration $c$ using Algorithm \ref{ftda} with no data assimilation ($\mu=0$), $P_2$ on the same fine mesh as for the Stokes FE problem, and $\Delta t=0.02$.  Plots of the DNS solution are shown in figure \ref{dacontour} on the right side.

\begin{figure}[!ht]
\begin{center}
DA (t=0.0) \hspace{2in} DNS (t=0.0)\\
\includegraphics[width = .48\textwidth, height=.13\textwidth,viewport=50 20 650 300, clip]{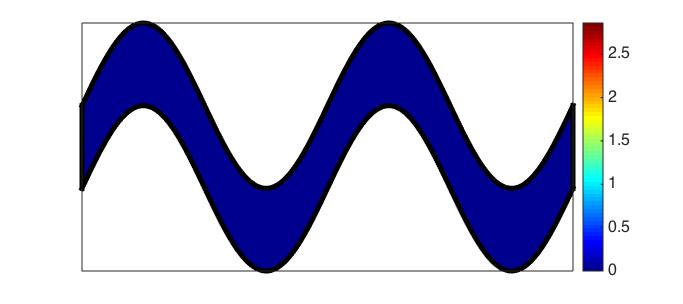}
\includegraphics[width = .48\textwidth, height=.13\textwidth,viewport=50 20 650 300, clip]{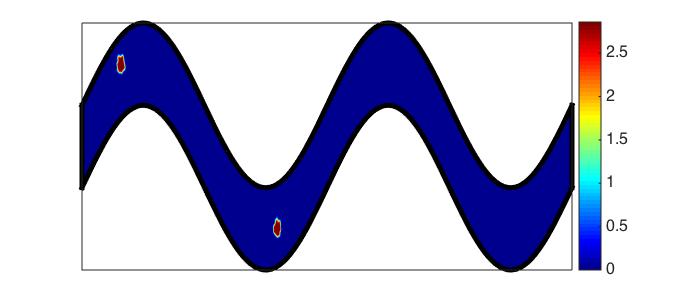}
DA (t=0.5) \hspace{2in} DNS (t=0.5)\\
\includegraphics[width = .48\textwidth, height=.13\textwidth,viewport=50 20 650 300, clip]{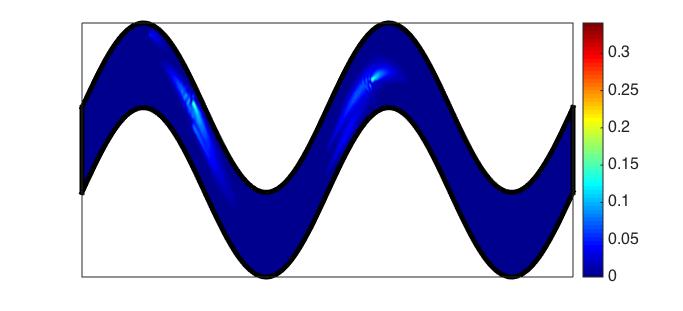}
\includegraphics[width = .48\textwidth, height=.13\textwidth,viewport=50 20 650 300, clip]{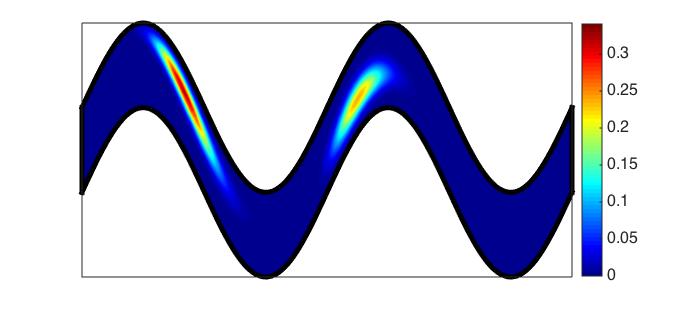}
DA (t=1.0) \hspace{2in} DNS (t=1.0)\\
\includegraphics[width = .48\textwidth, height=.13\textwidth,viewport=50 20 650 300, clip]{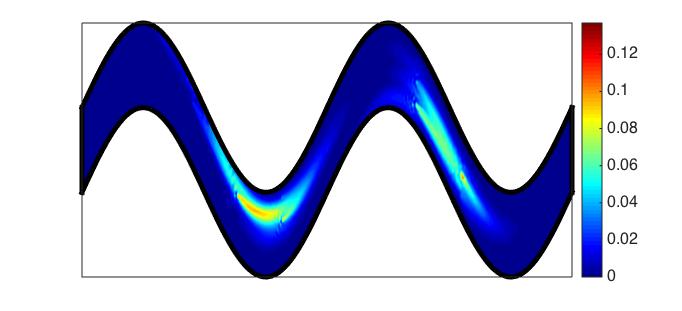}
\includegraphics[width = .48\textwidth, height=.13\textwidth,viewport=50 20 650 300, clip]{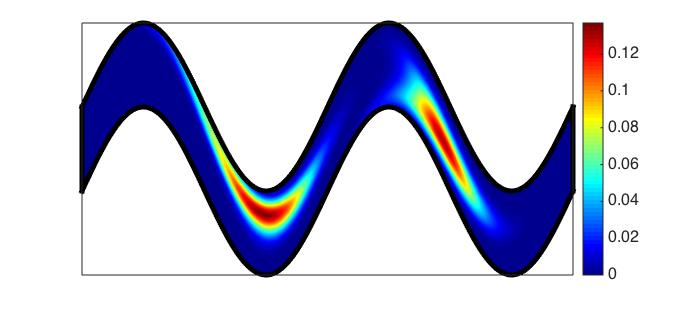}
DA (t=2.5) \hspace{2in} DNS (t=2.5)\\
\includegraphics[width = .48\textwidth, height=.13\textwidth,viewport=50 20 650 300, clip]{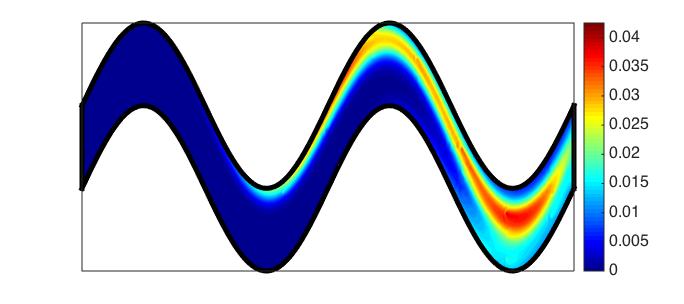}
\includegraphics[width = .48\textwidth, height=.13\textwidth,viewport=50 20 650 300, clip]{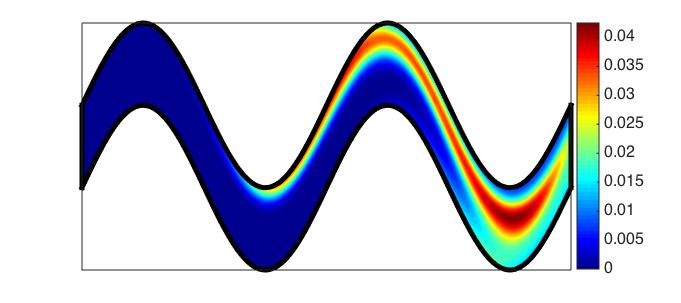}
DA (t=4.0) \hspace{2in} DNS (t=4.0)\\
\includegraphics[width = .48\textwidth, height=.13\textwidth,viewport=50 20 650 300, clip]{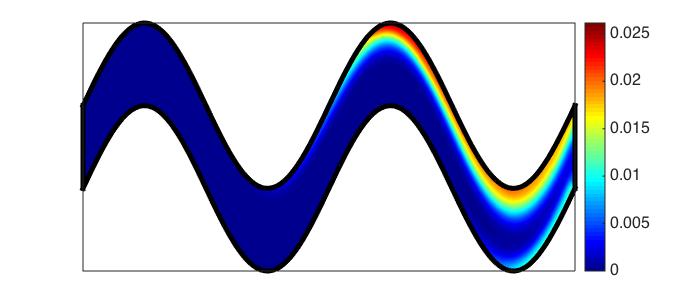}
\includegraphics[width = .48\textwidth, height=.13\textwidth,viewport=50 20 650 300, clip]{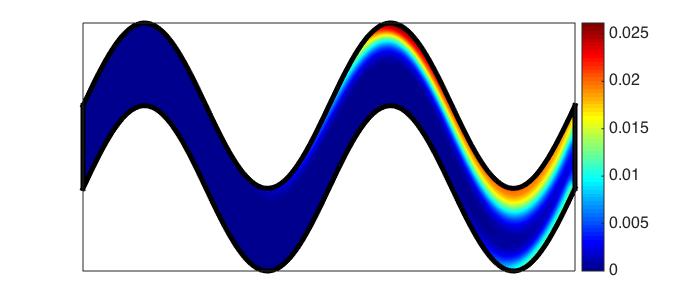}
	\caption{\label{dacontour} Contour plots of DA and DNS velocity magnitudes at times 0, 0.5, 1, 2.5, and 4.}
	\end{center}
\end{figure}

Finally, we compute the DA solution, using Algorithm \ref{ftda} with the same parameters as the DNS except zero initial conditions, taking the DNS solution $c$ as the true solution, and testing the algorithm using several choices of $\mu=\{0.01,0.1,1,10,100,1000\}$.   For the coarse mesh, we (purposely) choose a crude and simple mesh to show the robustness of the DA scheme, making a rectangular grid of $[0,4\pi]\times[-1,2]$, and intersecting it with the domain (see figure \ref{meshes} at bottom).  The nodes $x_{k_j}$ for the coarse mesh are the nodes from the fine mesh that are in element $E_j^H$ and closest to its center, and so for some elements, this node is on the boundary.

Convergence of the DA solution to the true (DNS) solution is shown in figure \ref{conv}, for varying $\mu$, in relative $L^2$ norms of the difference (relative norms are used since the true solution decays significantly over this time period).  We observe almost identical convergence for $\mu$=1, 10, 100, and 1000.  The DA solution for $\mu=0.1$ seems to also converge, but at a slower rate.  Convergence for $\mu=0.01$ is even slower, but does still appear to be converging.  

Note that if $\mu=0$, the relative error will always be one since the DA solution will always be 0 (0 initial condition, no forcing, homogenous  Dirichlet and Neumann boundary conditions).  The absolute error will go to zero for large enough $t$, but this corresponds to the case of all contaminant leaving the river through the outflow or finally diffusing away.  Hence waiting for the solution to assimilate with no nudging is not useful.

For the case $\mu=100$ (which is very closely resembled by the solutions to $\mu=$1, 10, and 1000), we show contour plots of the DA and true (DNS) solution at t=0, 0.5, 1, 2.5 and 4 in figure \ref{dacontour}.  We observe agreement between the solutions increasing, until finally by t=4 the solutions are visually indistinguishable.

\section{Application: Data assimilation in incompressible Navier-Stokes equations}

We consider now application of DA with the new interpolant for the incompressible NSE, which are given by
\begin{align}
u_t + (u\cdot \nabla )u - \nu \Delta u + \nabla p  &= f, \label{nse1}
\\ \nabla \cdot u & = 0, \label{nse2} 
\end{align}
where $u$ represents velocity, $p$ pressure, viscosity $\nu>0$, and external forcing $f$.  The system is also equipped with an initial condition $u_0$, and for simplicity of analysis we consider no-slip velocity boundary conditions and zero-mean pressure.

The associated DA scheme we consider uses an IMEX BDF2 temporal discretization, finite element spatial discretization, and uses our
new proposed efficient interpolant.  We assume the velocity-pressure finite element
spaces $(X_h,Q_h)=((P_k)^d,P_{k-1})$ satisfy the inf-sup stability condition (if the pressure space is
discontinuous, then an appropriate mesh must be chosen for the inf-sup condition to hold \cite{qin:p1:p0,Z11b,Z05,JLMNR17,GuzmanNeilan13}{\color{black})}.

Define the discretely divergence-free space by
\[ 
V_h := \{ v_h \in X_h \,\, |\,\, (\nabla \cdot v_h, q_h) = 0 \,\, \forall \,\, q_h \in Q_h   \},
\]
and the usual NSE trilinear form $b:X\times X\times X\rightarrow \mathbb{R}$ by
\[
b(u,v,w)=\frac12 (u\cdot\nabla v,w) - \frac12 (u\cdot\nabla w,v).
\]
The scheme reads as follows.

\begin{algorithm} \label{bdf2alg1}
	Given any initial conditions $v_h^0,\ v_h^{1} \in V_h$, forcing $f \in L^\infty(0,\infty; L^2(\Omega))$, true solution $ u \in L^\infty(0,\infty; L^2(\Omega))$, and nudging parameter $\mu>0$, find $(v_h^{n+1}, q_h^{n+1})$ $\in$ $(X_h, Q_h)$ for $n = 1,2,...$, satisfying
	\begin{eqnarray}
	\frac{1}{2\Delta t} \left( 3v_h^{n+1} - 4v_h^n + v_h^{n-1},\chi_h \right) + b(2v_h^{n} - v_h^{n-1},v_h^{n+1},\chi_h) - (q_h^{n+1},\nabla \cdot \chi_h) && \nonumber \\ + \nu (\nabla v_h^{n+1},\nabla \chi_h) + \mu (\tilde{P}_{L^2}^H (v_h^{n+1} - u^{n+1}),\tilde{P}_{L^2}^H \chi_h)&= & (f^{n+1},\chi_h), \label{femda1bdf2} \\
	(\nabla \cdot v_h^{n+1},r_h)  &=& 0 , \label{femda2bdf2}
	\end{eqnarray}
	for all $(\chi_h, r_h) \in X_h \times Q_h$.
\end{algorithm}

\subsection{Analysis of the DA algorithm for NSE}

We now state a lemma for long-time stability and well-posedness.  In our analysis, we will use the parameter $\alpha := \nu - C\mu H^2$, where $C$ depends on the size of the true solution.  We will require that $\alpha>0$, which can be thought of as the coarse mesh $H$ being fine enough.

\begin{lemma} \label{L2stabilityBDF2}
Assume $\alpha>0$.  Then  for any $\Delta t>0$, Algorithm \ref{bdf2alg1} is well-posed globally in time, and solutions are nonlinearly long-time stable: for any $n>1$,
	\begin{align*}
& \bigg( {\color{black}C_u^{-2}} \left( \| v_h^{n+1} \|^2 + \| v_h^n \|^2 \right)  + \frac{\alpha\Delta t}{4} \| \nabla v_h^{n+1} \|^2 + \frac{\mu\Delta t}{4}\| v_h^{n+1} \|^2 \bigg) \\
\le&
 \left( {\color{black}C_l^{-2} (\| v_h^1 \|^2 + \| v_h^0 \|^2)}  + \frac{\alpha\Delta t}{4} \| \nabla v_h^{1} \|^2 + \frac{\mu\Delta t}{4}\| v_h^{1} \|^2 \right) \left( \frac{1}{1+\lambda\Delta t} \right)^{n+1}  + C\lambda^{-1} (\nu^{-1}F^2 + \mu U^2) .
\end{align*}
where $\lambda=\min \{\frac{\mu {\color{black}C_l^2}}{4},\frac{\alpha C_P^{-2}{\color{black}C_l^2}}{4}, 2\Delta t^{-1} \}$, and $F:=\| f \|_{L^{\infty}(0,\infty;H^{-1})}$, $U := C\|\tilde{P}_{L^2}^Hu^{n+1}\|$.
\end{lemma} 

\begin{proof}
Well-posedness and long-time stability was proven in Lemma 3.4 of \cite{LRZ18} for a similar algorithm, except  with a different treatment of the nudging term.  This proof can be adapted with just minor modifications, using the analysis of the nudging term in the proof of Lemma \ref{mainlemma} above, to immediately provide the long time stability result.  With this established, well-posedness follows directly.
\end{proof}

We now prove that solutions to Algorithm \ref{bdf2alg1} converge to the true NSE solution at an optimal rate in the $L^2$ norm, globally in time, provided restrictions on $\dt$ and $\mu$ are satisfied. The time derivative term will again be handled with the G-stability theory in a manner similar to the stability proof.

\begin{theorem} \label{bdf2conv2}
	Let $u,p$ solve the NSE \eqref{nse1}-\eqref{nse2} with given $f\in L^{\infty}(0,\infty;L^2(\Omega))$ and $u_0\in L^2(\Omega)$, with $u \in L^\infty (0, \infty; H^{k+1}(\Omega))$, $p \in L^\infty (0, \infty; H^{k}(\Omega))$ ($k\ge 1$), $u_{tt} \in L^{\infty}(0,\infty;L^2(\Omega))$, and $u_{ttt} \in L^\infty(0, \infty; H^1(\Omega))$.  Denote $U:= | u |_{L^{\infty}(0,\infty;H^{k+1})}$ and $P:= | p |_{L^{\infty}(0,\infty;H^{k})}$.  Assume the time step size satisfies
\[
\Delta t <  CM^2\nu^{-1} \left( h^{2k-3} U^2  +   \|\nabla u^{n+1}\|^2_{L^3} + \|u^{n+1}\|_{L^\infty}^2 \right)^{-1},
\]
and the parameter $\mu$ satisfies
\[
 CM^2\nu^{-1} \bigg(  h^{2k-3} U^2 +  \|\nabla u^{n+1}\|_{L^3}^2 + \| u^{n+1}\|_{L^{\infty}}^2  \bigg) < \mu < \frac{2\nu}{C H^2}.
\]
Then if the boundary is sufficiently smooth so that the discrete Stokes projection has optimal approximation properties, the error in solutions to Algorithm \ref{bdf2alg1} satisfies, for any $n$,	
\[
\| u^n - v_h^n \|^2 \le C\left( \frac{1}{1+\lambda\Delta t} \right)^{n} \| u_0 - v_h^0 \|^2 + \frac{R}{\lambda},
\] 
where 
	$
	R =C(\Delta t^4 + H^2h^{2k} )$
		and $\lambda = 2\alpha {\color{black} C_l^2}  C_P^{-2}$.	
\end{theorem}

\begin{proof}
	The proof of this theorem follows similar to Theorem 3.7 in \cite{LRZ18}, but with some minor modifications, in particular using the treatment of the nudging term from the previous section, and (as pointed out in \cite{GNT18}) using the discrete Stokes projection in the definition of the interpolation error term $\eta$ instead of the $L^2$ projection into $V_h$.
	\end{proof}

\subsection{Numerical experiments for incompressible NSE}

\begin{figure}[!ht]
	\centering
	\includegraphics[scale = .8]{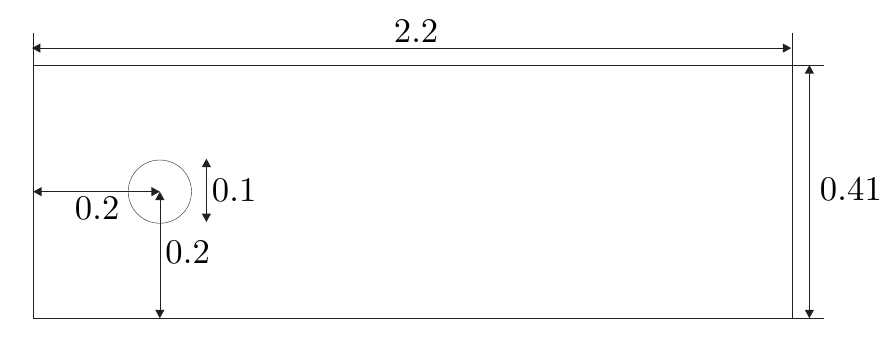}
	\caption{\label{cylinderdomain} Shown above is the domain for the flow past a cylinder test problem.}
\end{figure}

To test the DA algorithm for incompressible NSE, we consider Algorithm \ref{bdf2alg1} applied to 2D channel flow 
past a cylinder \cite{ST96}.  We will consider Reynolds numbers $Re=100$ and $Re=500$.  The domain is the rectangular channel  {\color{black}[0, 2.2]}$\times$[0, 0.41], with a cylinder centered at $(0.2,0.2)$ with radius $0.05,$ see Figure \ref{cylinderdomain}. There is no external forcing ($f=0$), no-slip boundary conditions are prescribed for the walls and the cylinder, an inflow profile is given by 
\begin{align*}
u_1(0,y,t) & = u_1(2.2,y,t) = \frac{6}{0.41^2}y(0.41-y),
\\u_2(0,y,t) & = u_2(2.2,y,t) = 0.
\end{align*}
For Re=100, we take $\nu=0.001$, and for Re=500, we take $\nu=0.0002$.

We prescribe different outflow boundary conditions for the two cases.  For $Re=100$, we enforce the Dirichlet condition that the outflow be the same as the inflow, and for $Re=500$, we use the zero-traction boundary condition and enforce it with the usual `do-nothing' condition.  The nonlinear term is also treated differently, as no skew symmetry is used. Thus the $Re=500$ test does not fit the assumptions of our analysis above (which assumes full Dirichlet boundary conditions), and the difference is important since the nonlinear terms that vanish in our analysis will no longer vanish (additional boundary integrals will arise, even if divergence-free elements are used).  Still, channel flow with no stress / no traction outflow conditions is important in practice since Dirichlet outflow is not physical for higher Reynolds numbers, and thus this is an important practical test for DA algorithms.

Since we do not have access to a true solution for this problem, we instead use computed solutions.  They are obtained using Algorithm \ref{bdf2alg1} but without nudging (i.e. $\mu=0$), using $(P_2,P_1^{disc})$ Scott-Vogelius elements on barycenter refined Delaunay meshes that provide 35K velocity dof for Re=100, and 103K velocity dof for Re=500, a time step of $\Delta t=0.002$, and with the simulations starting from rest ($u_h^0=u_h^{1}=0$).  We will refer to these solutions as the DNS solutions.  Lift and drag calculations were performed for the computed solution and compared to the literature \cite{ST96,XMRT18}, which verified the accuracy of the {\color{black}DNSs}.

For the lift and drag calculations, we used the definitions
\begin{align*}
c_d(t) &= 20\int_S \left( \nu \frac{\partial u_{t_S}(t)}{\partial n}n_y - p(t)n_x \right)dS,
\\c_l(t) &= 20 \int_S \left( \nu \frac{\partial u_{t_S}(t)}{\partial n}n_x - p(t)n_y \right)dS,
\end{align*}
where $p(t)$ is pressure, $u_{t_S}$ is tangential velocity $S$ is the cylinder, and $n = \langle n_x, n_y\rangle$ the outward unit normal to the domain. For the calculations, we used the global integral formulation from \cite{J04}.

The coarse meshes for DA are constructed using the intersection of uniform rectangular meshes with the domain.  We take $H$ to be the width of each rectangle, and use several choices of $H$ in our tests.  Figure \ref{cylmesh} shows in red the 35K dof mesh and associated $H=\frac{2.2}{8}$ coarse mesh in black.

\begin{figure}[!ht]
	\centering
	\includegraphics[width = .50\textwidth,height=.12\textwidth]{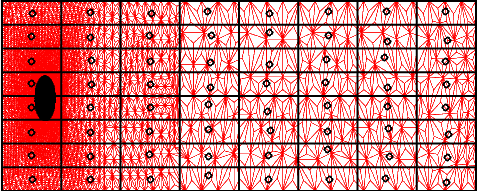}	
	\caption{\label{cylmesh} Shown above is the FE mesh (in red) and the $H=\frac{2.2}{8}$ coarse mesh and nodes (in black).}
\end{figure}

For the DA computations, we start from zero initial conditions $v_h^1 = v_h^0=0$, use the same spatial and temporal discretization parameters as the DNS for that Reynolds number, and start assimilation with the t=5 DNS solution (i.e., time 0 for DA corresponds to t=5 for the DNS).   The simulations are run on [5,10], and with varying $\mu$ and $H$.

\subsubsection{Re=100}

\begin{figure}[!ht]
	\centering
	\includegraphics[width = .50\textwidth]{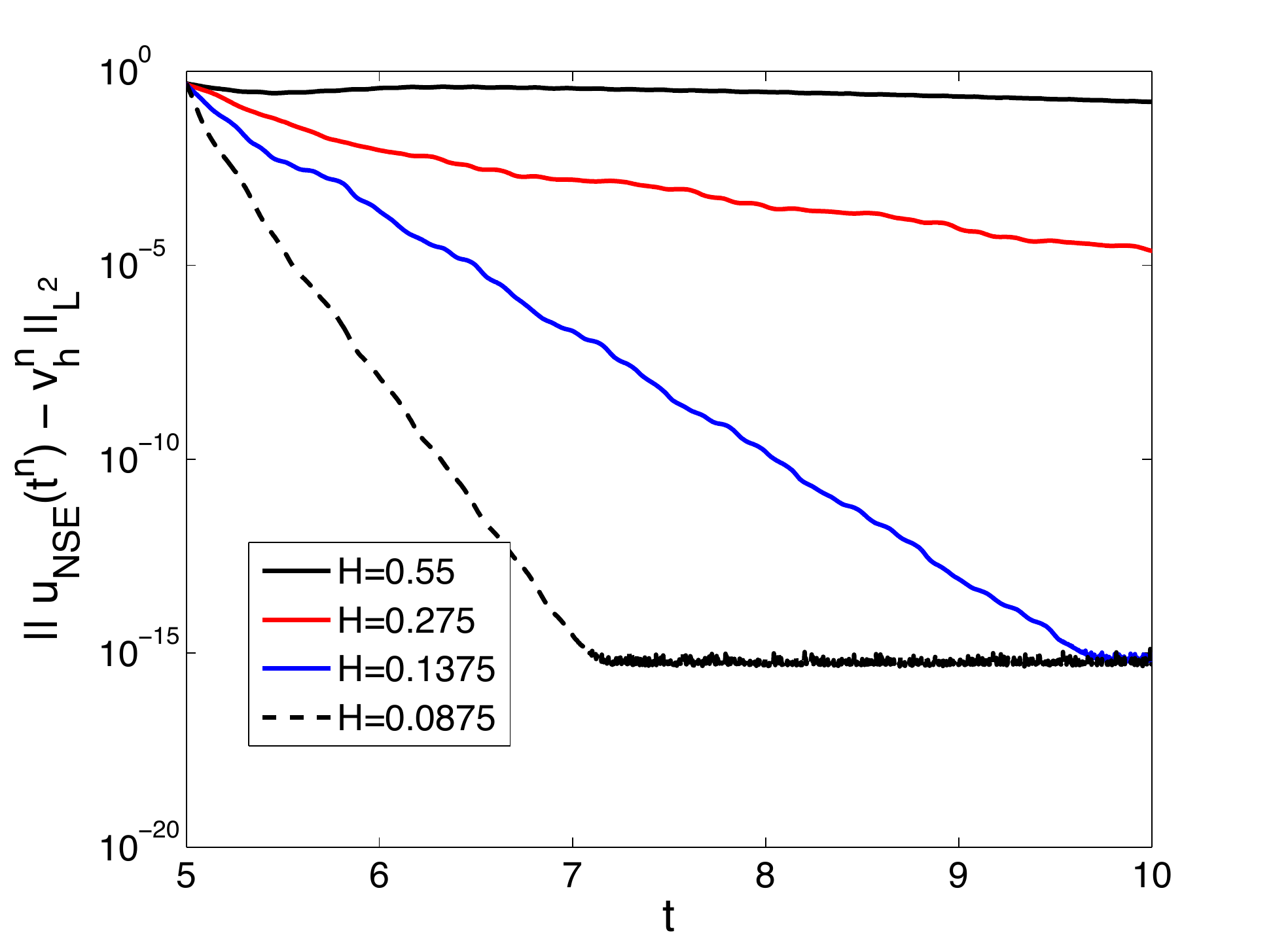}	
	\caption{\label{differenceL2} Shown above is the $L^2$ difference between the DA and DNS solutions versus time, for varying $\mu$ and $H$.}
\end{figure}

\begin{figure}[!ht]
	\centering
	\includegraphics[width = .49\textwidth,height=.3\textwidth]{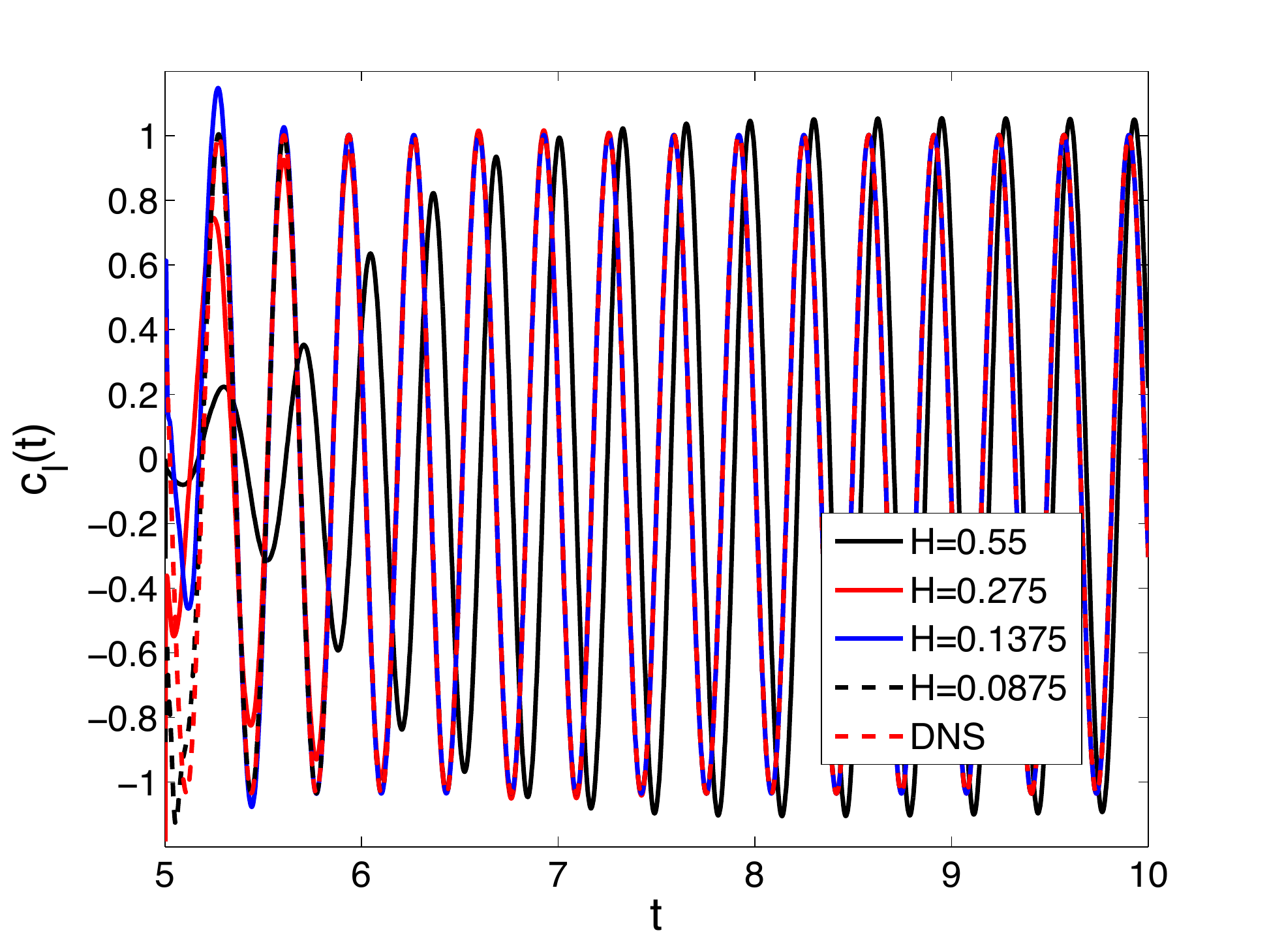}	
	\includegraphics[width = .49\textwidth,height=.3\textwidth]{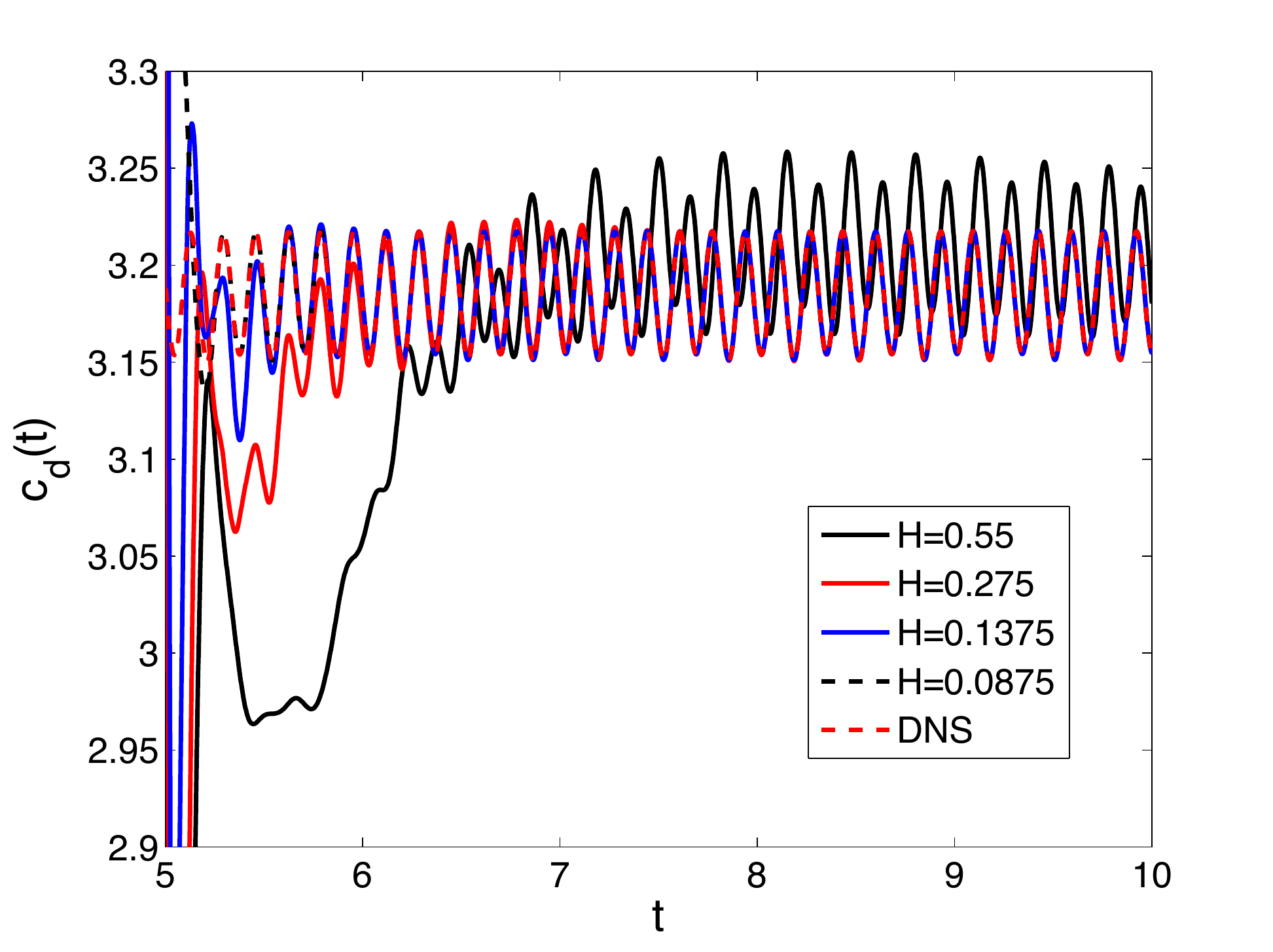}	
	\caption{\label{difference} Shown above are the lift and drag coefficients for Re=100 simulations for DA ($\mu=10$) with varying $H$, and for the DNS.}
\end{figure}

\begin{figure}[!ht]
\begin{center}
DA, H=0.55 (t=5.0) \hspace{.9in} DA, H=0.1375 (t=5.0) \hspace{.75in} DNS (t=5.0) \ \ \ \ \ \\
\includegraphics[width = .32\textwidth, height=.13\textwidth,viewport=80 20 650 230, clip]{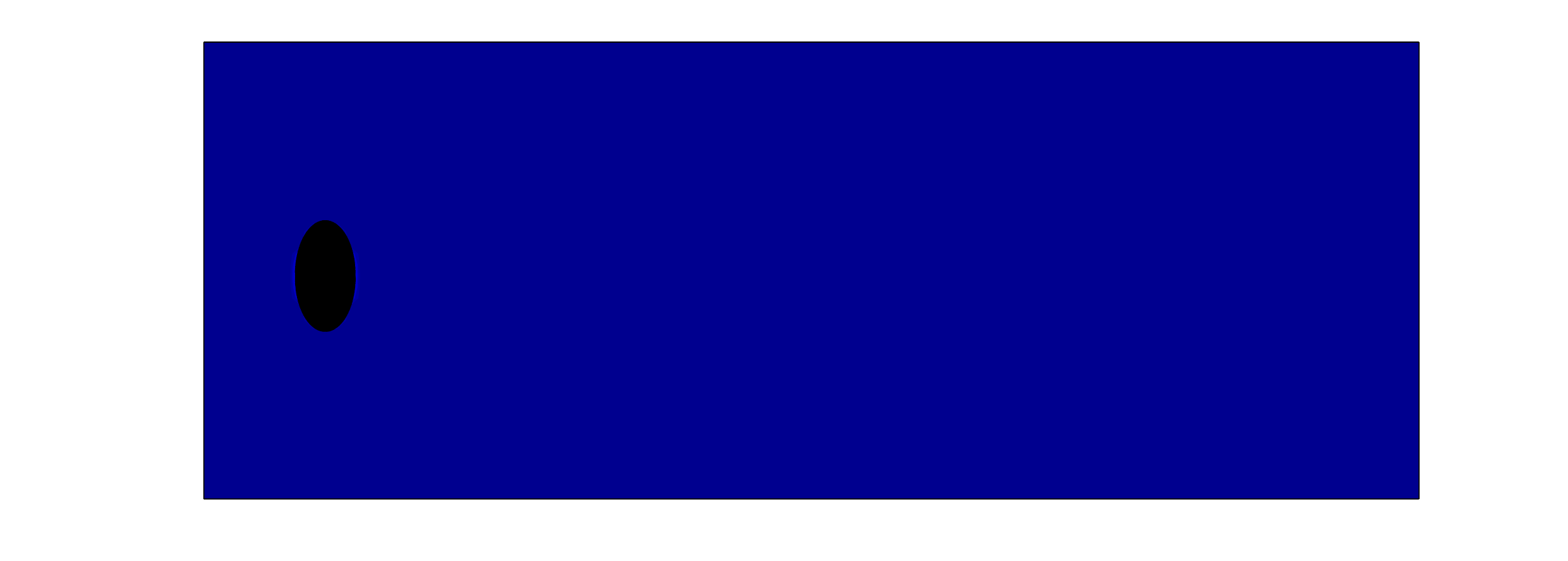}
\includegraphics[width = .32\textwidth, height=.13\textwidth,viewport=80 20 650 230, clip]{cylDA0.pdf}
\includegraphics[width = .32\textwidth, height=.13\textwidth,viewport=80 20 650 230, clip]{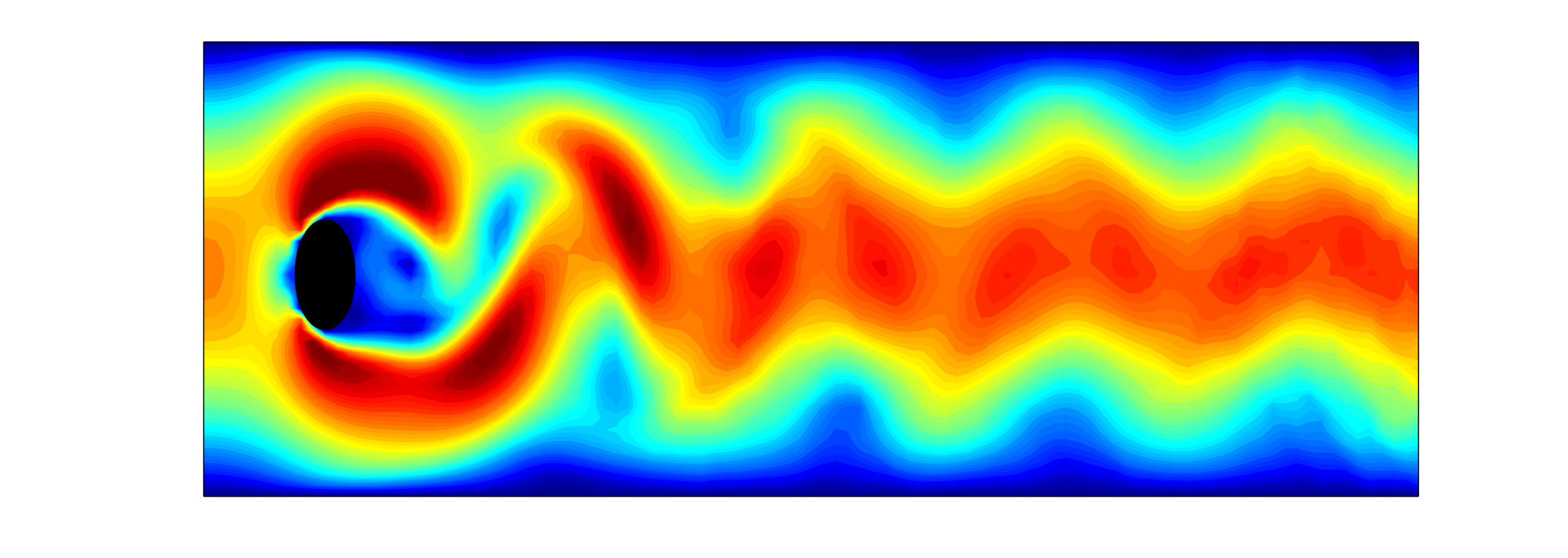}
DA, H=0.55 (t=5.5) \hspace{.9in} DA, H=0.1375 (t=5.5) \hspace{.75in} DNS (t=5.5) \ \ \ \ \ \\
\includegraphics[width = .32\textwidth, height=.13\textwidth,viewport=55 20 500 260, clip]{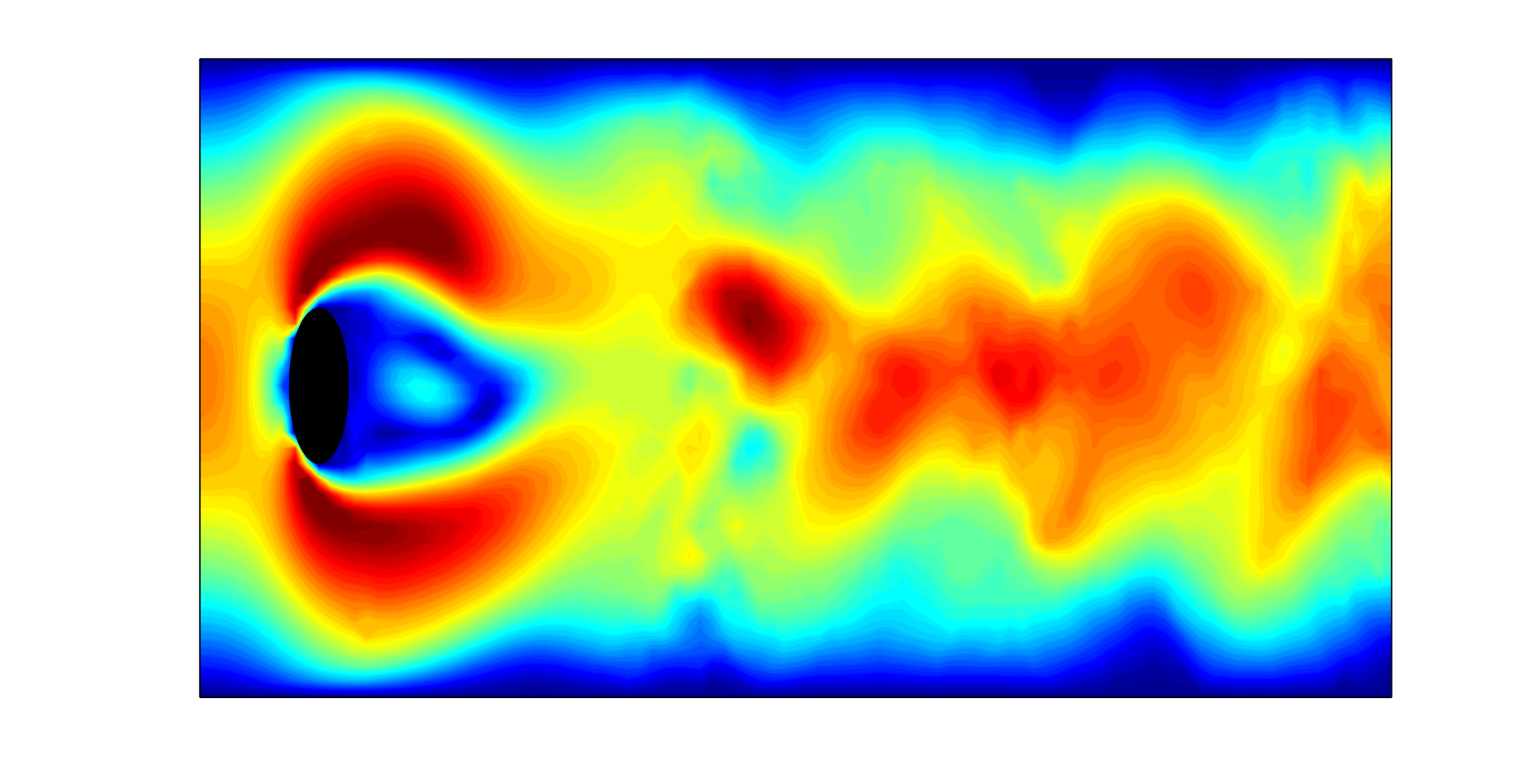}
\includegraphics[width = .32\textwidth, height=.13\textwidth,viewport=55 20 500 260, clip]{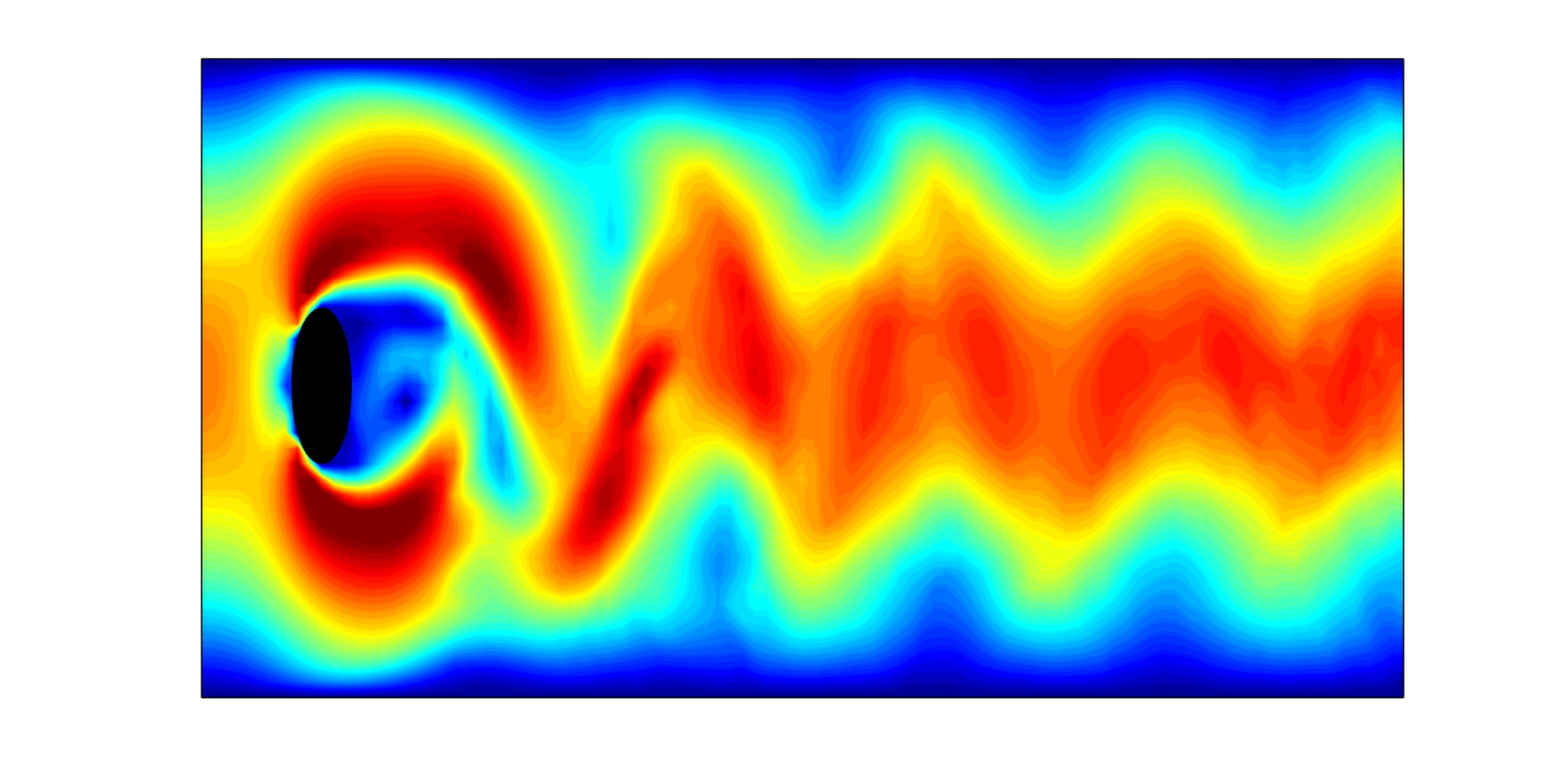}
\includegraphics[width = .32\textwidth, height=.13\textwidth,viewport=55 20 500 260, clip]{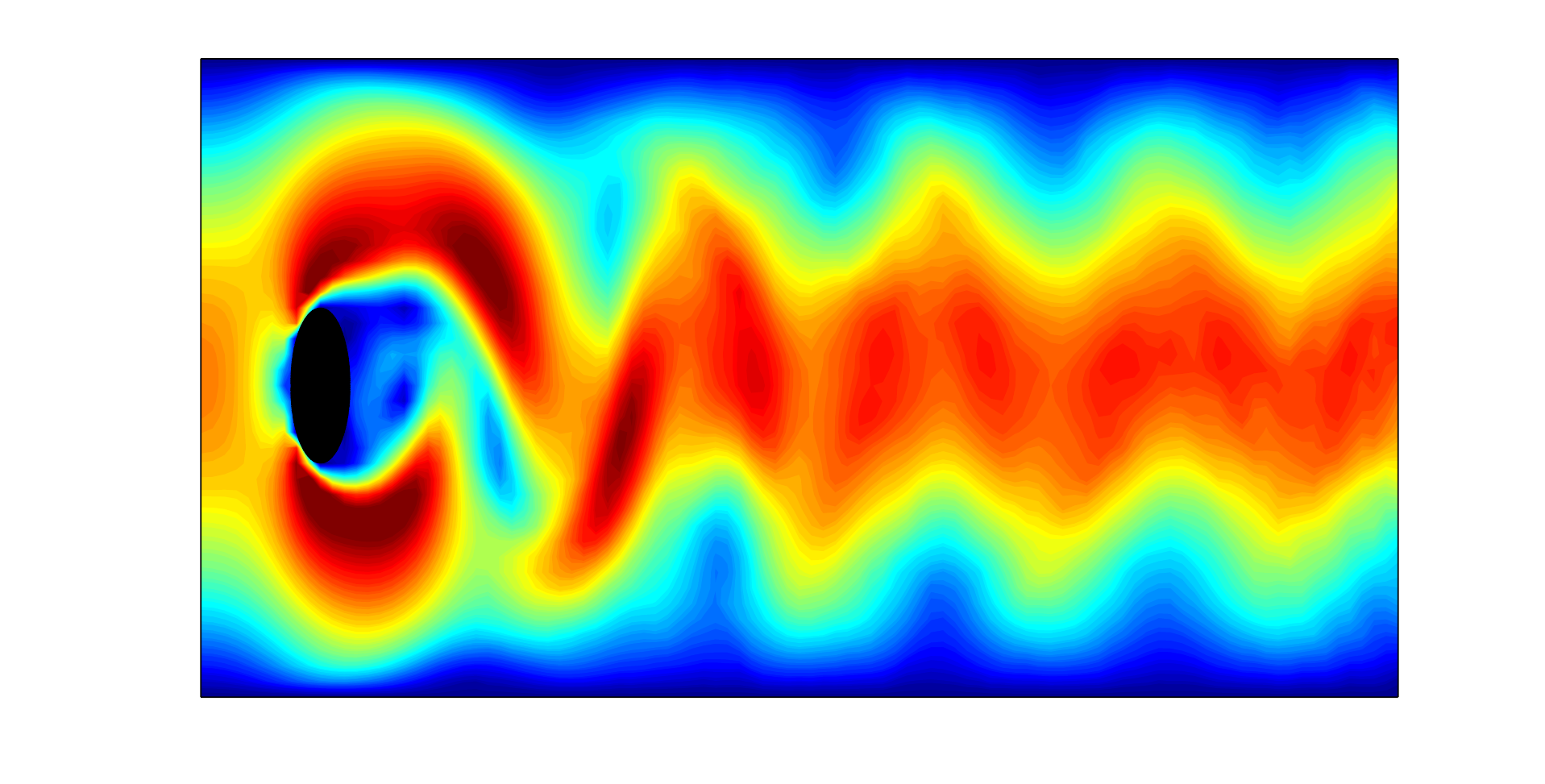}
DA, H=0.55 (t=6.0) \hspace{.9in} DA, H=0.1375 (t=6.0) \hspace{.75in} DNS (t=6.0) \ \ \ \ \ \\
\includegraphics[width = .32\textwidth, height=.13\textwidth,viewport=55 20 500 260, clip]{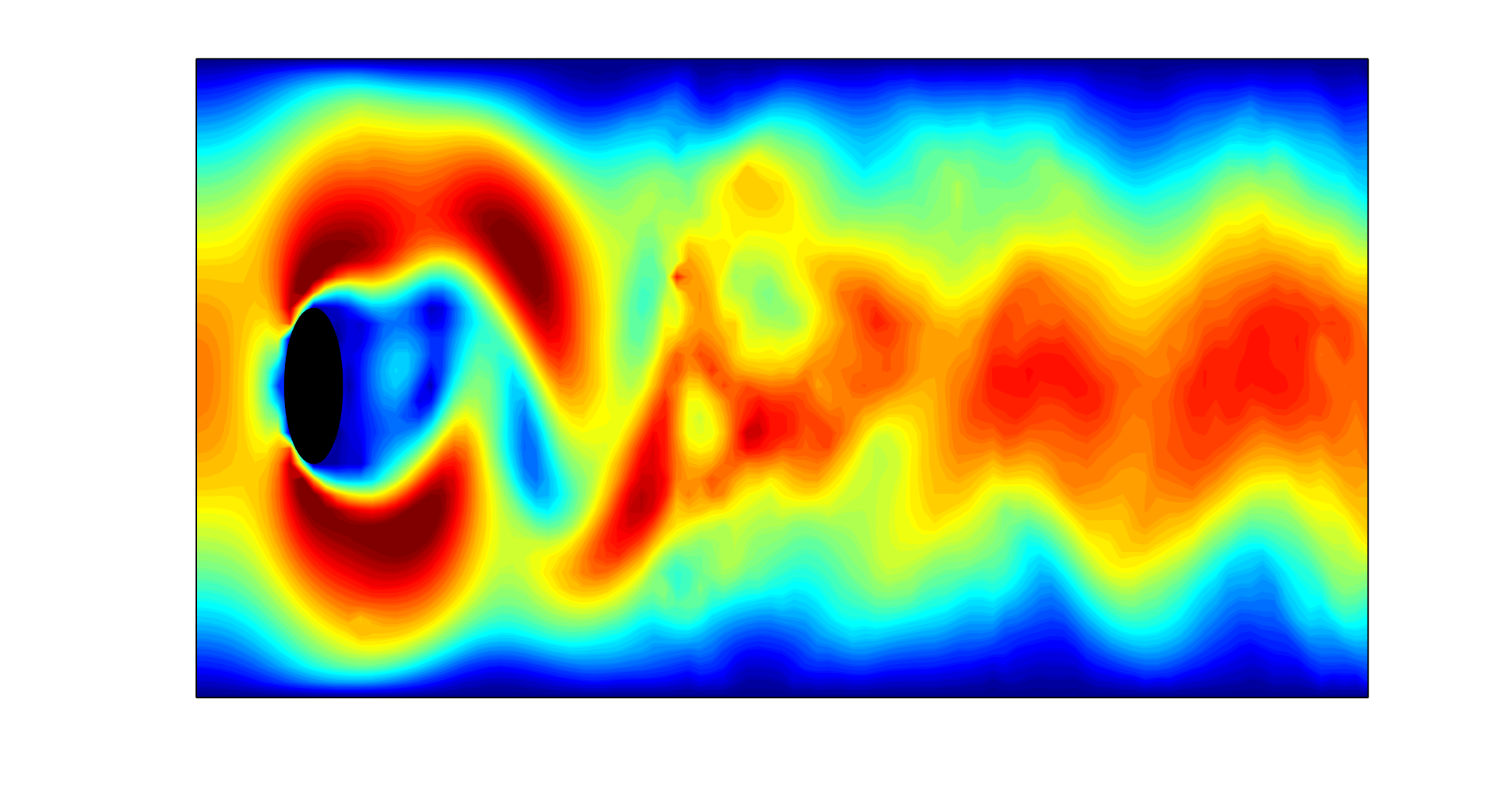}
\includegraphics[width = .32\textwidth, height=.13\textwidth,viewport=55 20 500 260, clip]{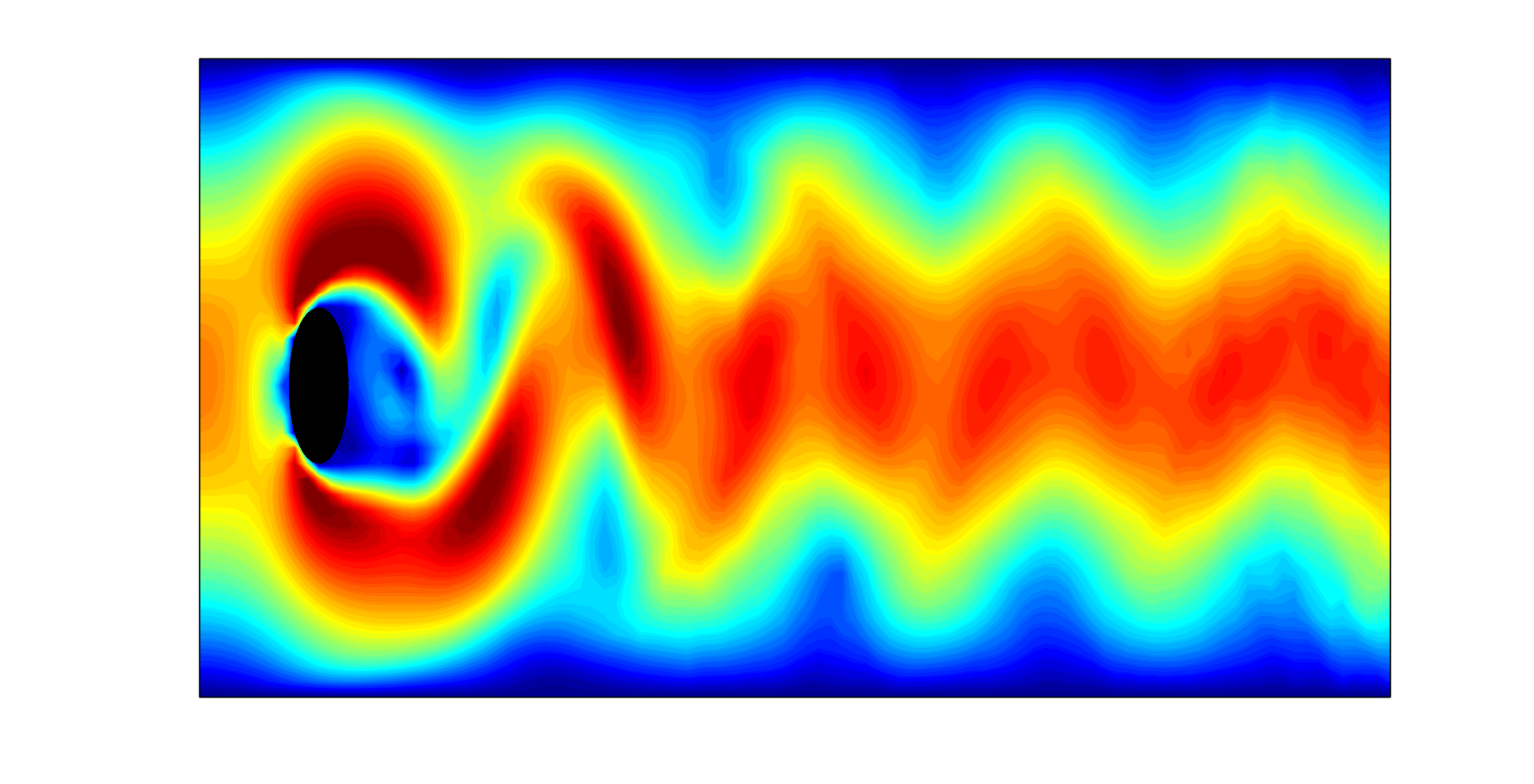}
\includegraphics[width = .32\textwidth, height=.13\textwidth,viewport=55 20 500 260, clip]{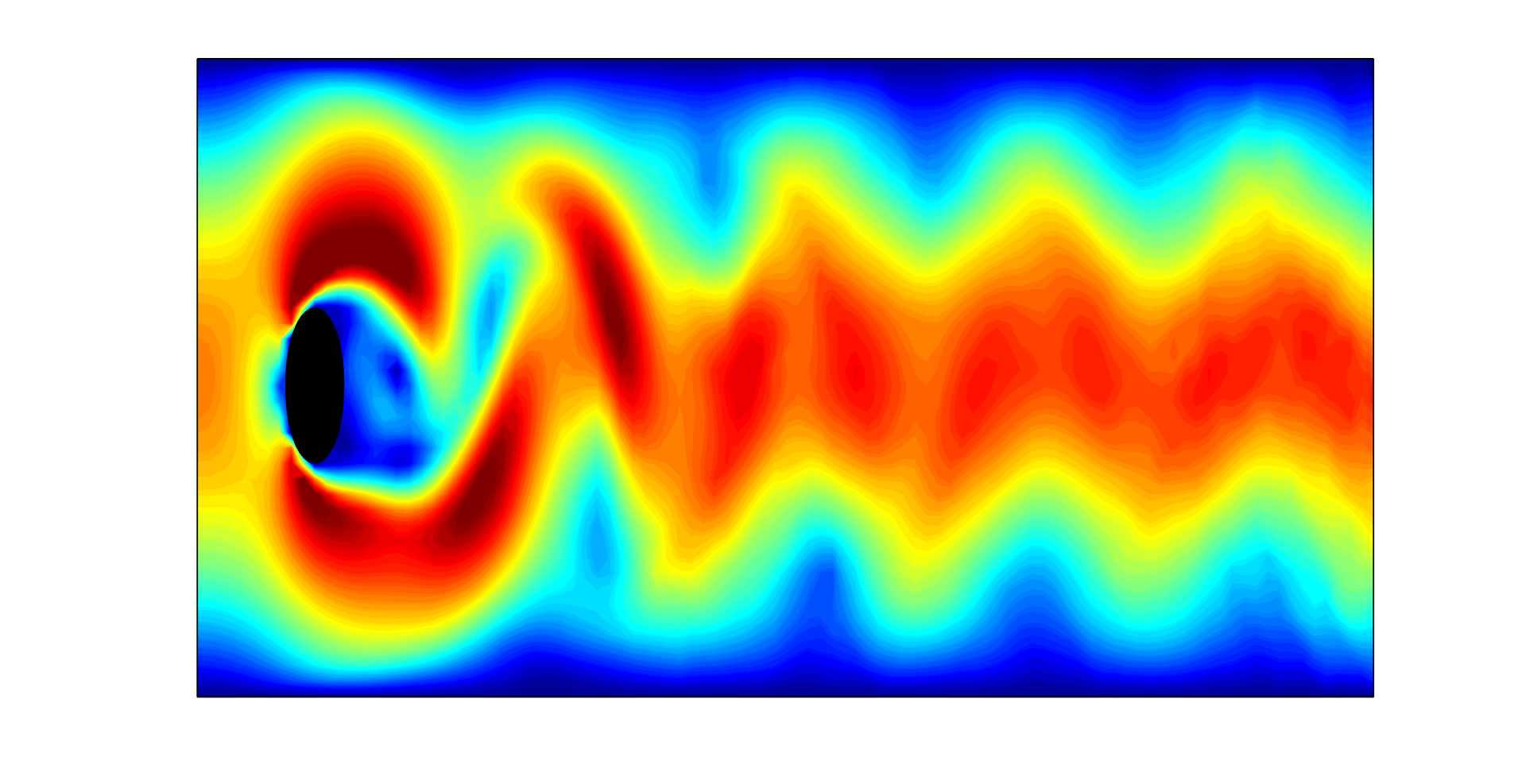}
DA, H=0.55 (t=10) \hspace{.9in} DA, H=0.1375 (t=10) \hspace{.75in} DNS (t=10) \ \ \ \ \ \\
\includegraphics[width = .32\textwidth, height=.13\textwidth,viewport=55 20 500 260, clip]{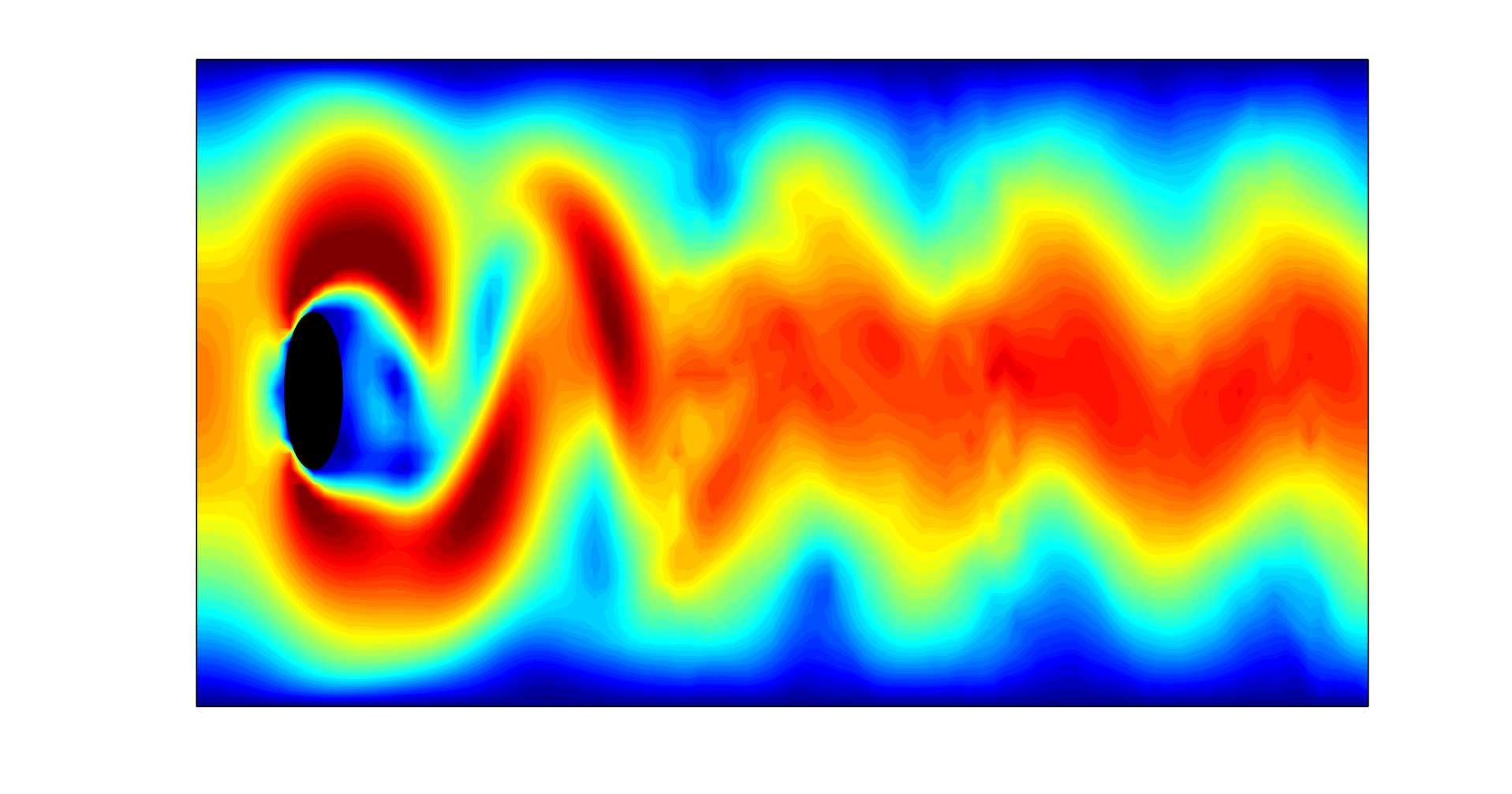}
\includegraphics[width = .32\textwidth, height=.13\textwidth,viewport=55 20 500 260, clip]{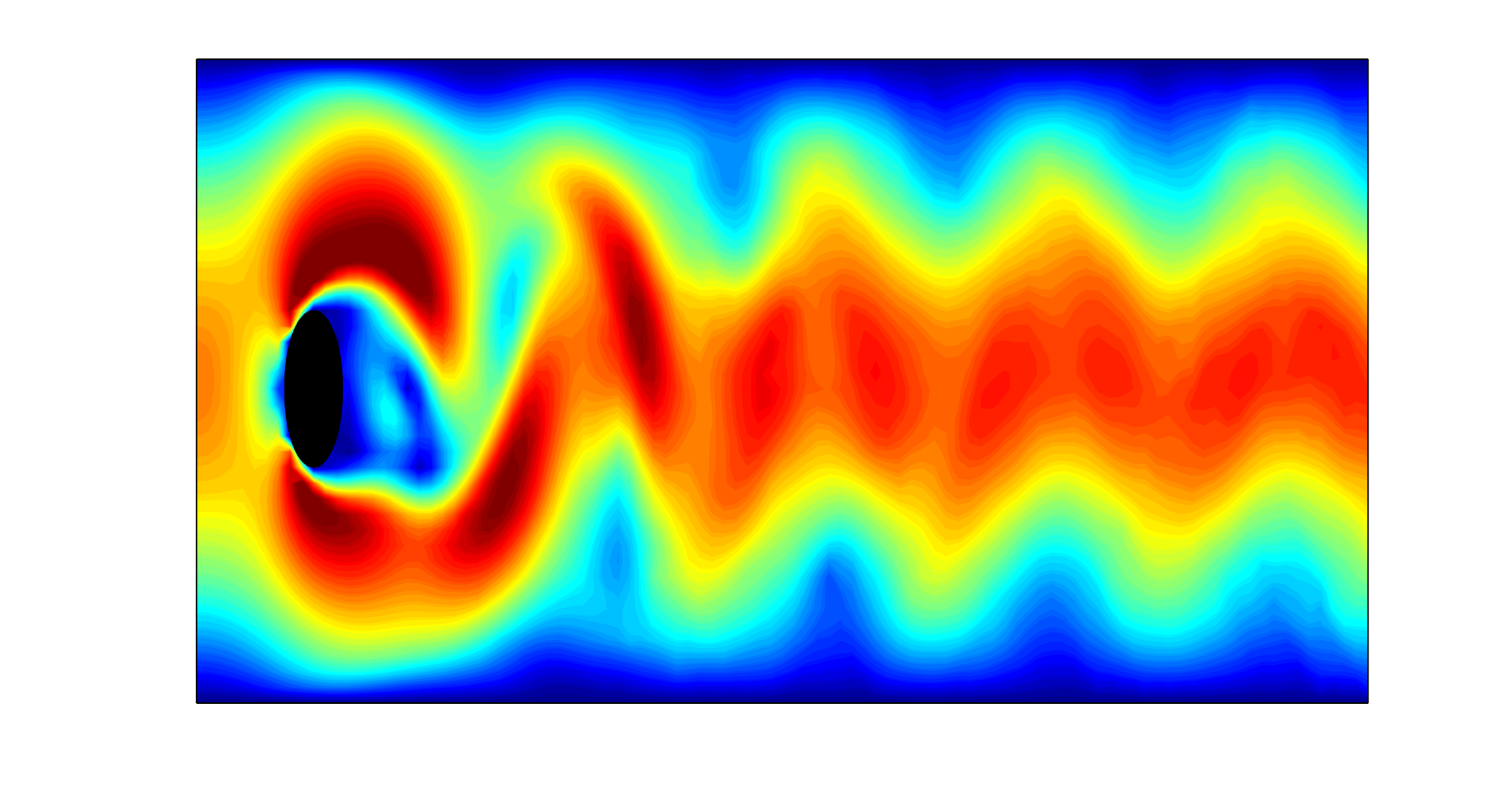}
\includegraphics[width = .32\textwidth, height=.13\textwidth,viewport=55 20 500 260, clip]{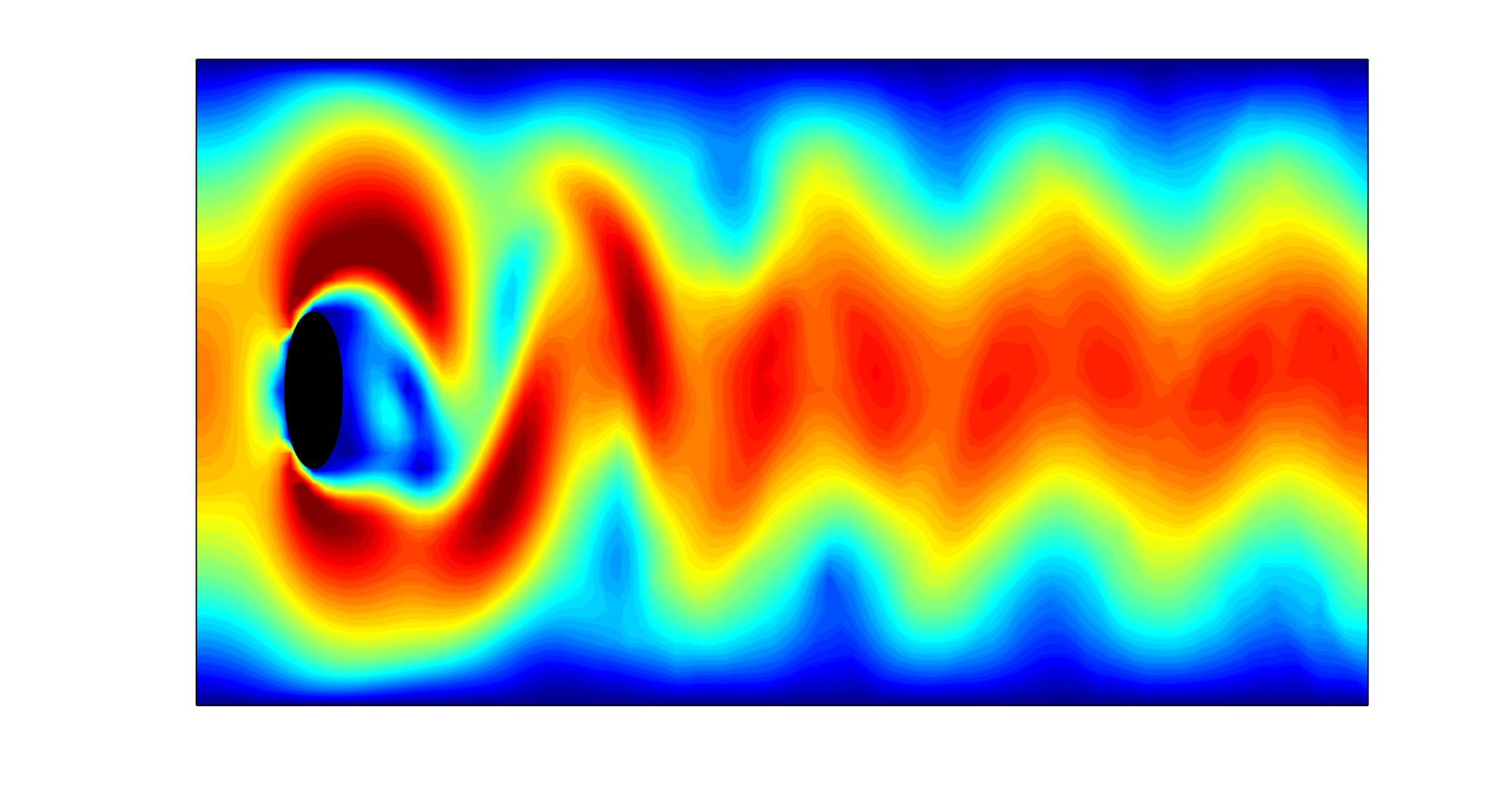}
	\caption{\label{contourcyl} Speed contour plots of DA solutions with $\mu=10$ with $H=0.55$ (left) and $H=0.1375$ (center), and DNS solutions, at times 5, 5.5, 6 and 10.}
	\end{center}
\end{figure}

Results are shown for the $Re$=100 tests in figures \ref{differenceL2}-\ref{contourcyl}.  For all $Re=100$ tests, we use $\mu=10$.  Figure \ref{differenceL2} shows convergence in time of the DA schemes to the DNS.  We observe that the DA solutions from the two finest $H$'s converge to the DNS, and quickly.  For $H=0.275$, the DA solution does appear to be converging, although slowly, and agrees with the DNS to $10^{-5}$ by t=10.  We do not observe convergence for $H=0.55$ in this time interval.  Plots of lift and drag coefficients versus time are shown in figure \ref{difference}, and all DA solution except for $H=0.55$ give good lift and drag predictions by t=7 (the solution from $H=0.55$ never gives good drag coefficient prediction{\color{black})}.

Figure \ref{contourcyl} shows speed contour plots of DA and DNS solutions at t=5 (the start time for DA), 5.5, 6, and 10.  The DA scheme with $H=0.1375$ (middle column) is already close to the DNS by t=5.5, and we observe no difference from the DNS by t=6.  The DA solution from $H=0.55$, on the other hand, does not converge by t=10.  At t=5.5 and t=6, it is clearly quite far from the DNS solution.  By t=10, it looks closer, but still shows significant differences from the DNS.

Overall, we observe good convergence of the DA solution to the DNS solution, provided the coarse mesh is fine enough.  However, `fine enough' is still quite coarse, as we observe good convergence even when only 64 measurement points ($H$=0.275) are used.

\subsubsection{Re=500}

\begin{figure}[!ht]
	\centering
	\includegraphics[width = .32\textwidth,height=.3\textwidth]{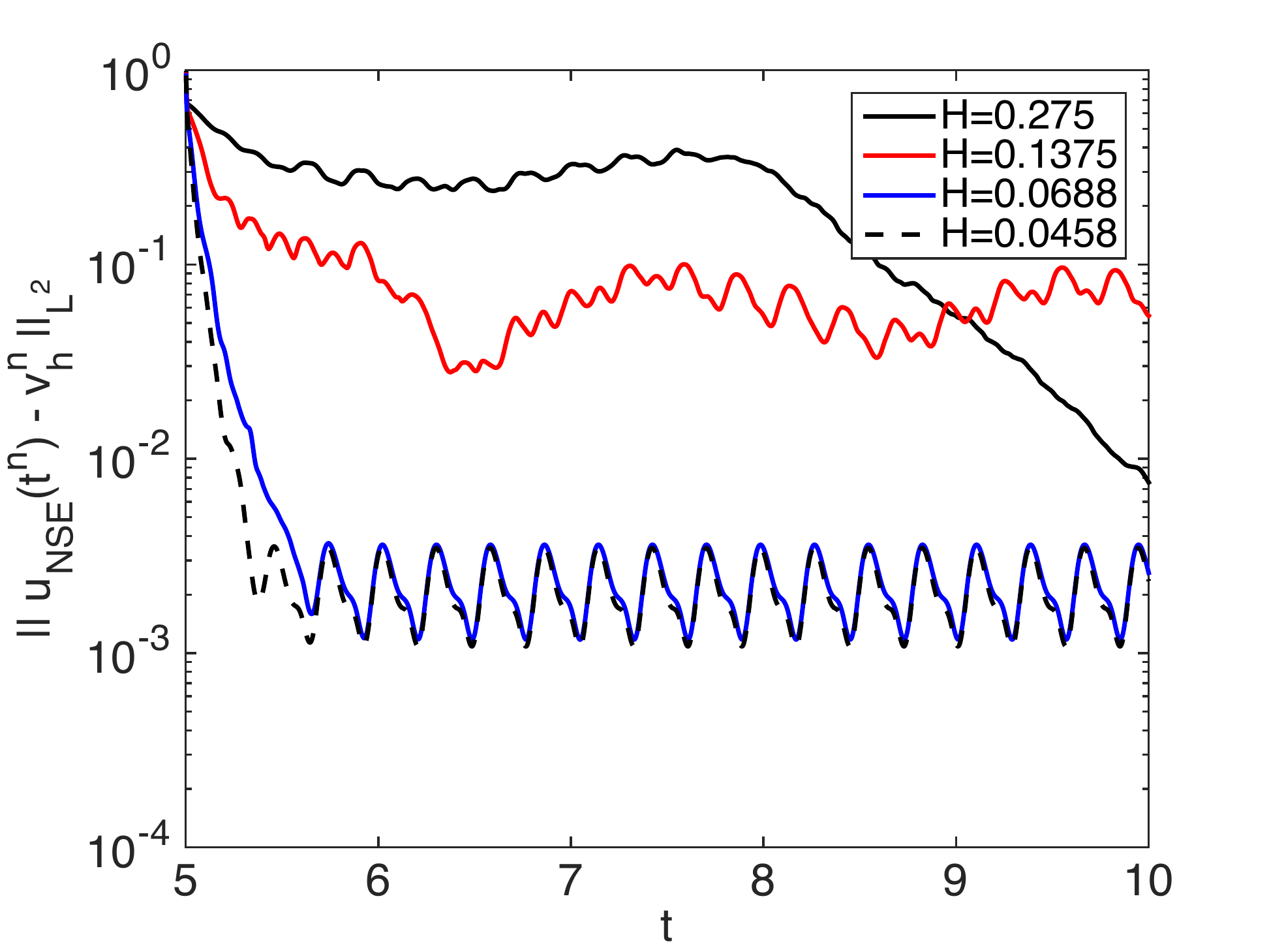}	
	\includegraphics[width = .32\textwidth,height=.3\textwidth]{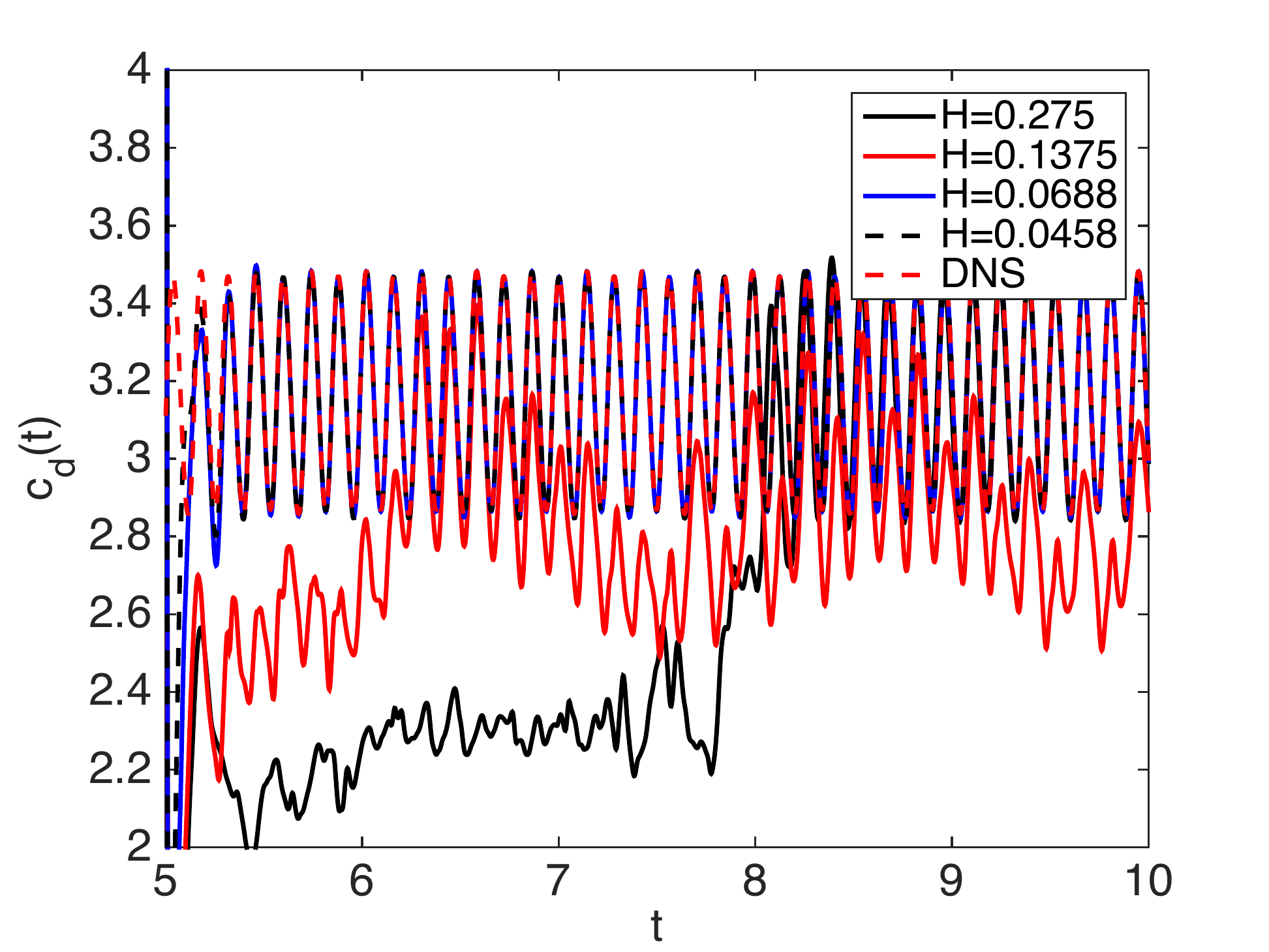}	
	\includegraphics[width = .32\textwidth,height=.3\textwidth]{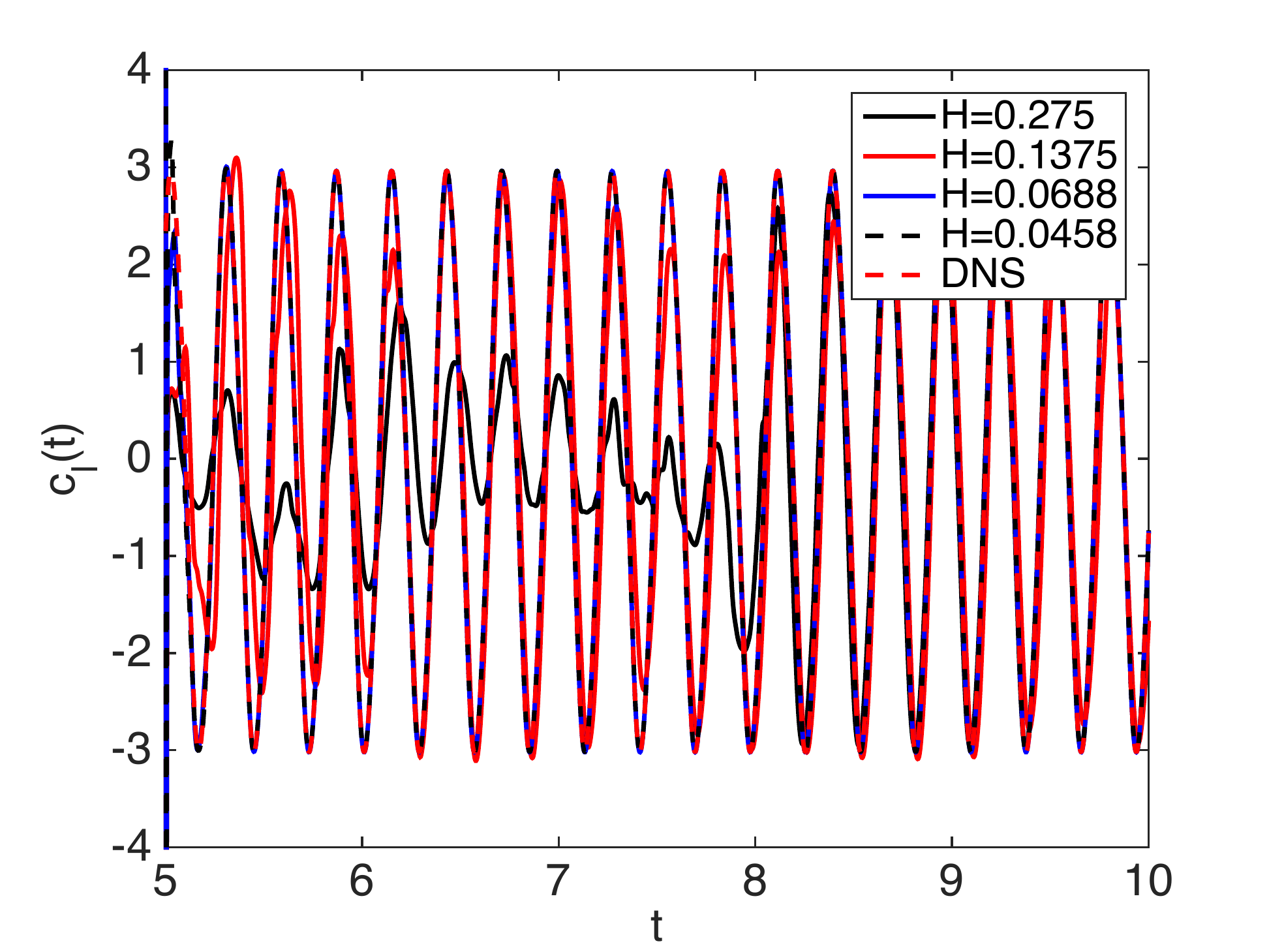}	
	\caption{\label{difference5002} Shown above are the lift and drag coefficient predictions for $Re=500$ simulations for DA with $\mu=10$ and varying $H$, and for the DNS.}
\end{figure}

\begin{figure}[!ht]
\begin{center}
DA, H=0.275 (t=6.0) \hspace{.7in} DA, H=0.0688 (t=6.0) \hspace{.75in} DNS (t=6.0) \ \ \ \ \ \\
\includegraphics[width = .32\textwidth, height=.15\textwidth,viewport=55 20 520 280, clip]{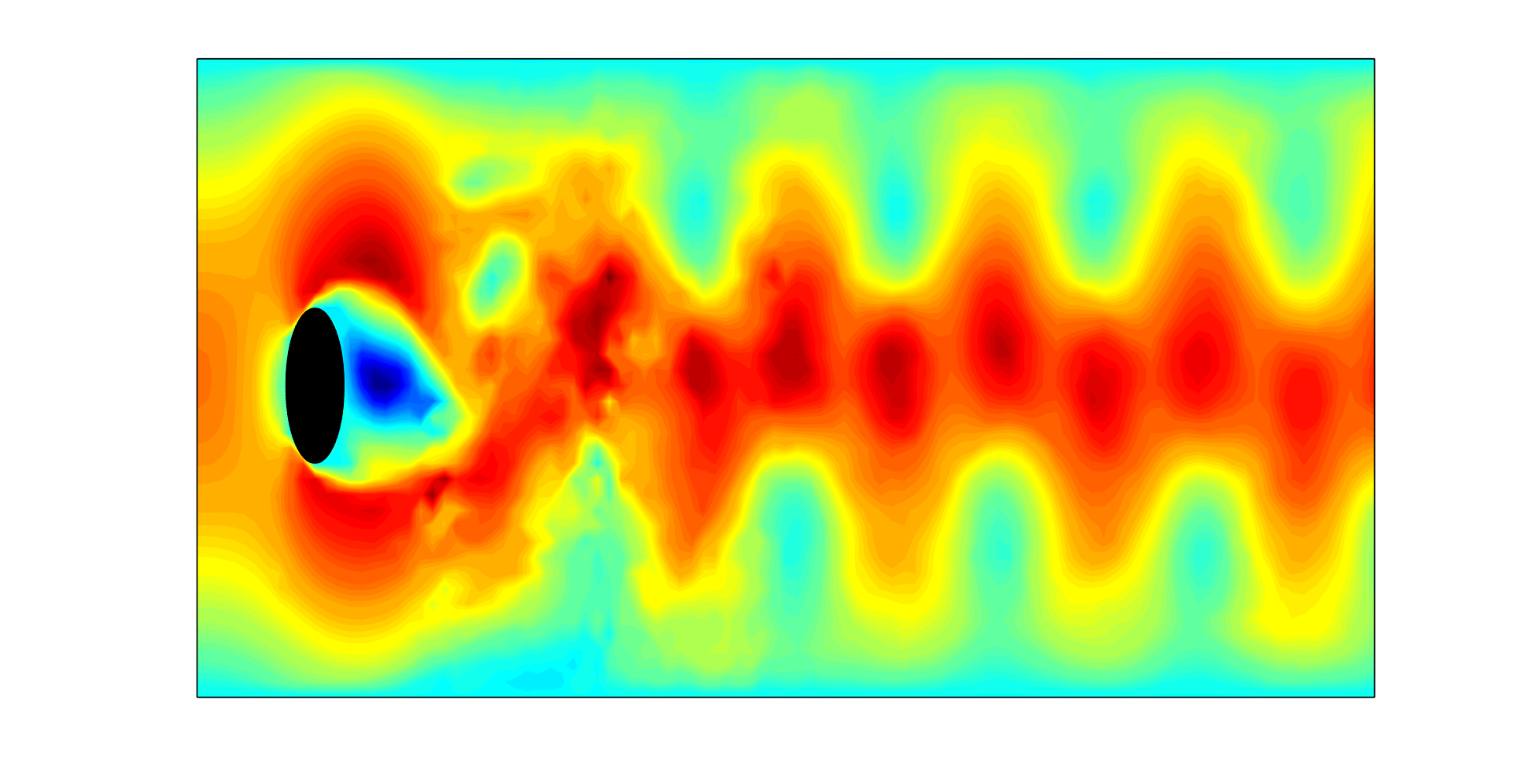}
\includegraphics[width = .32\textwidth, height=.15\textwidth,viewport=55 20 520 280, clip]{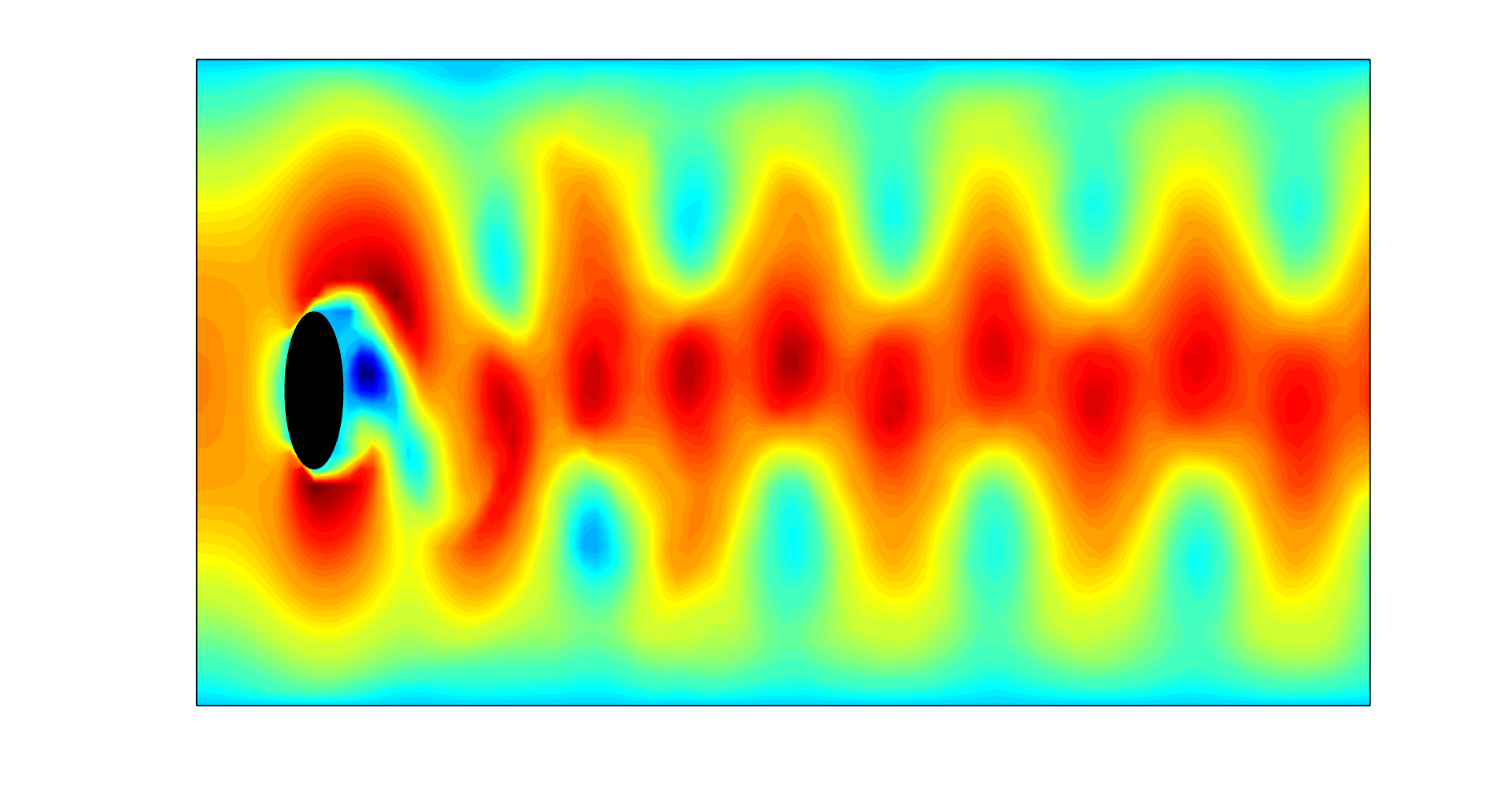}
\includegraphics[width = .32\textwidth, height=.15\textwidth,viewport=55 20 520 280, clip]{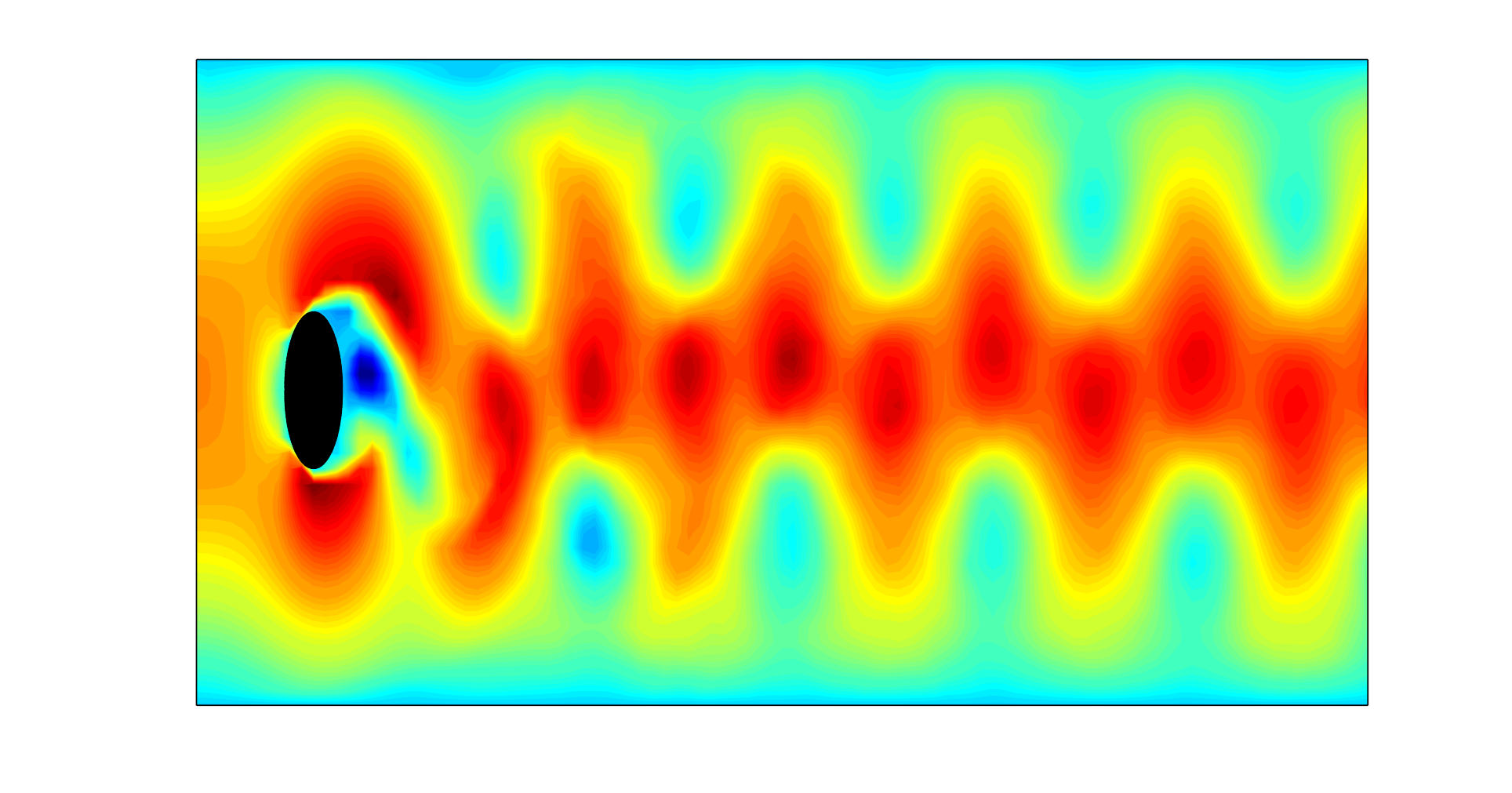}
DA, H=0.275 (t=10) \hspace{.7in} DA, H=0.0688 (t=10) \hspace{.75in} DNS (t=10) \ \ \ \ \ \\
\includegraphics[width = .32\textwidth, height=.15\textwidth,viewport=55 20 520 280, clip]{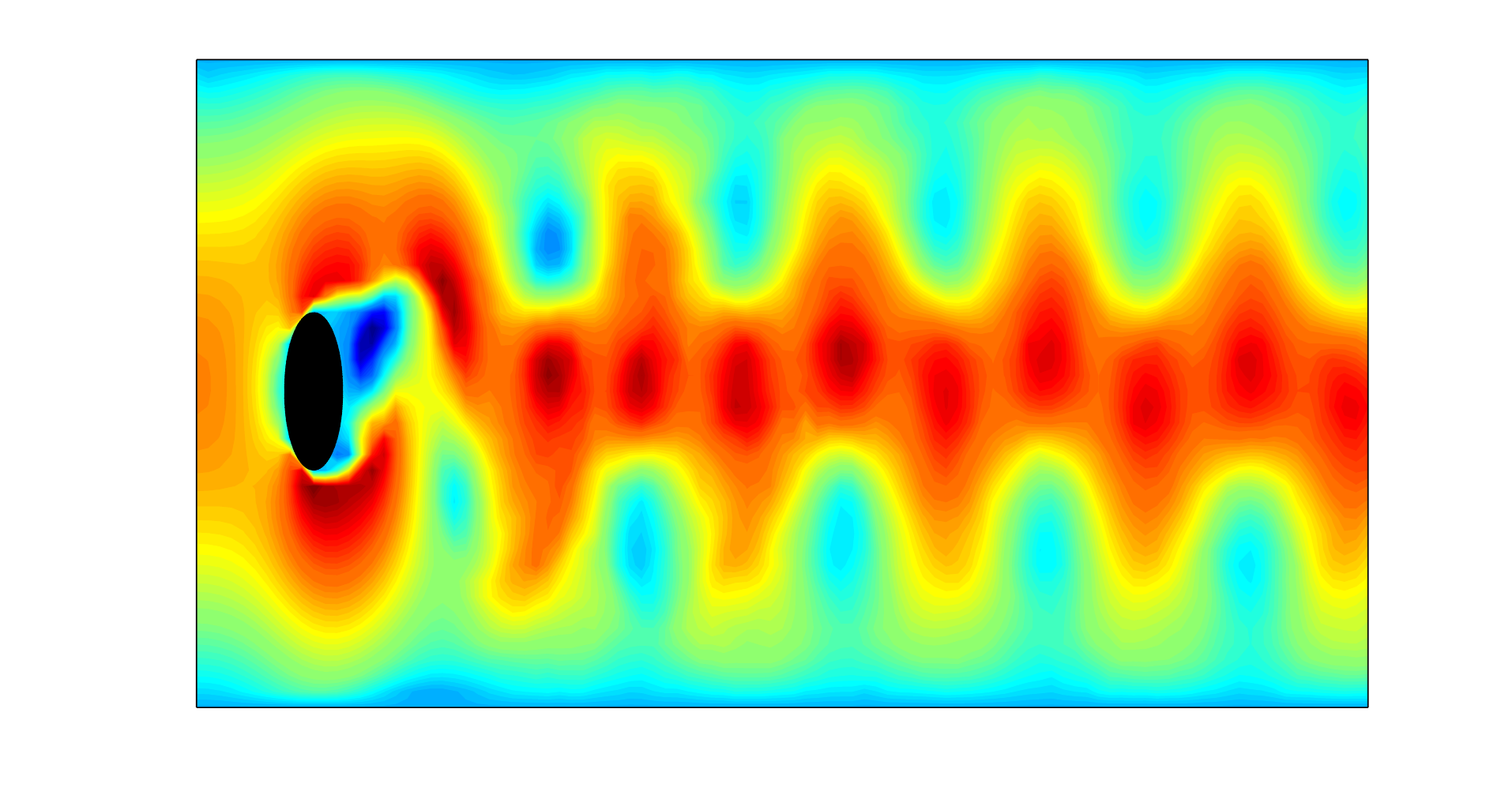}
\includegraphics[width = .32\textwidth, height=.15\textwidth,viewport=55 20 520 280, clip]{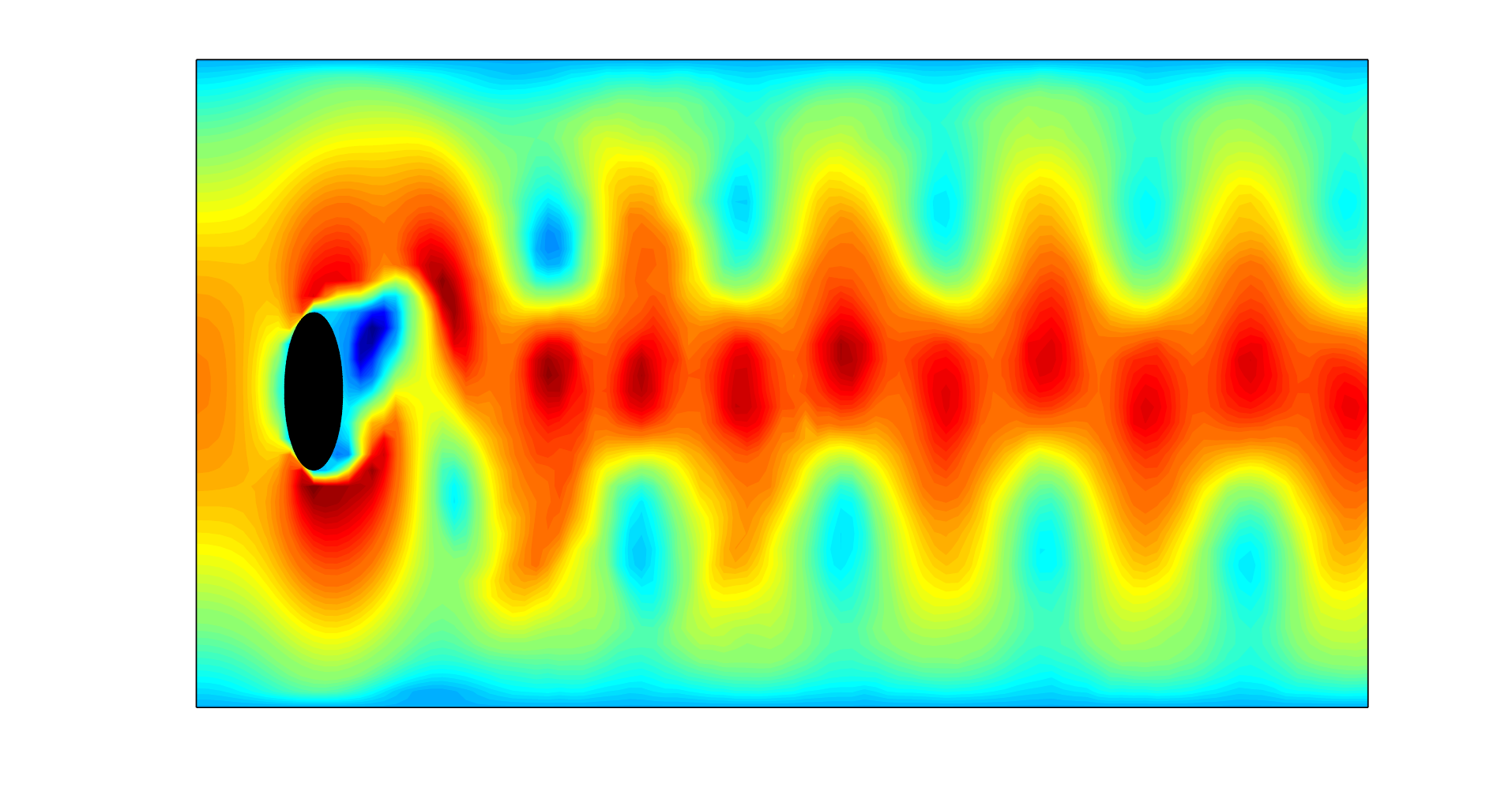}
\includegraphics[width = .32\textwidth, height=.15\textwidth,viewport=55 20 520 280, clip]{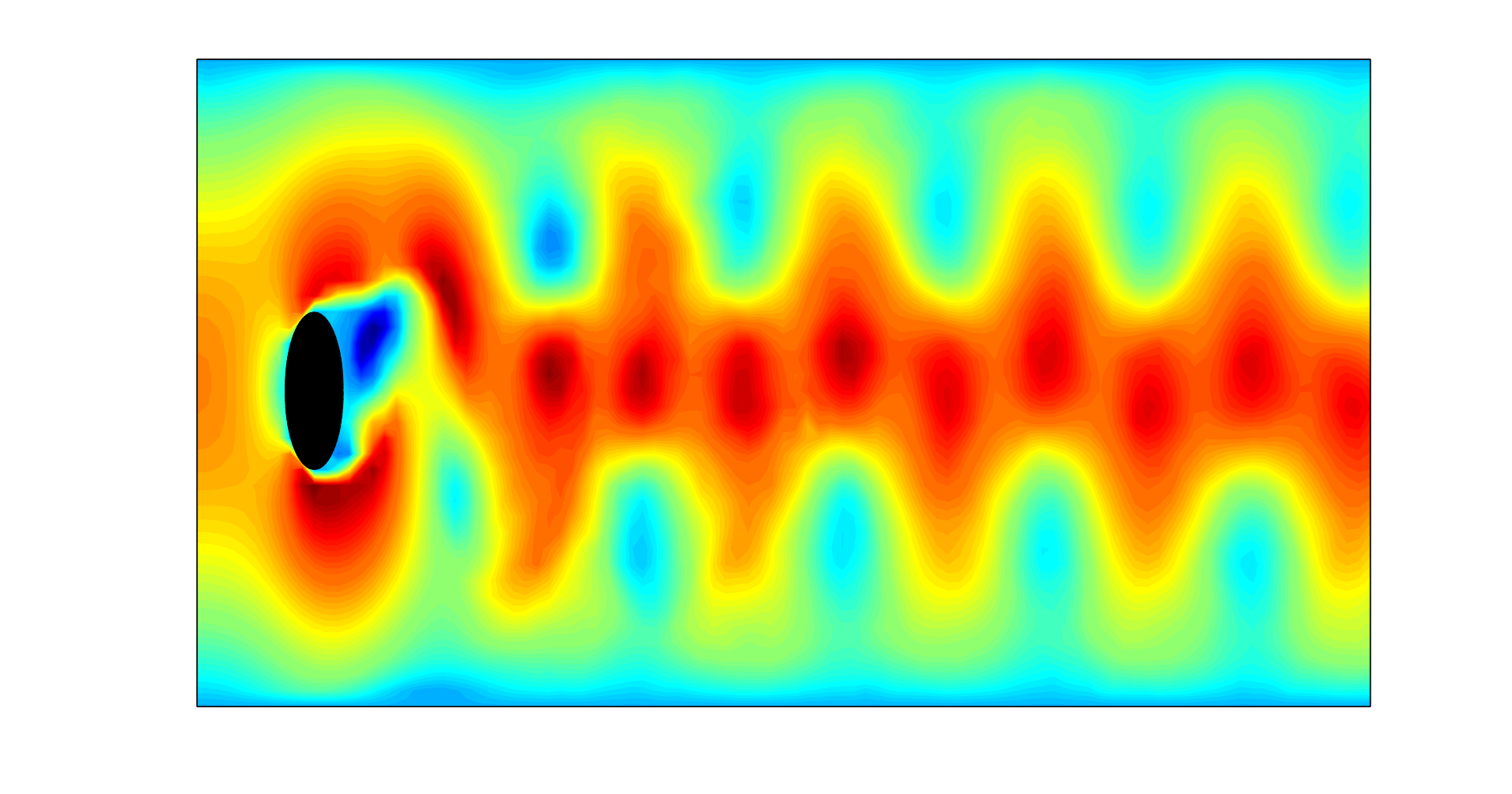}
	\caption{\label{contourcyl500} Speed contour plots of DA solutions for $Re=500$ with $\mu=10$, $H=0.275$ (left) and $H=0.0688$ (center), and DNS solutions, at times 6 and 10.}
	\end{center}
\end{figure}

We now give results for $Re=500$ numerical tests.  We remark again that due to the outflow boundary condition, the analysis in this section is not applicable, since the nonlinear terms behave in a different way.

Results for varying $H$ with $\mu=10$ are shown in figure \ref{difference500}, as $L^2$ error, and lift and drag coefficients.  An interesting phenomena is that the error appears to be bounded below, which does not happen in the $Re=100$ tests.  However, as we see in figures \ref{difference5002} and \ref{contourcyl500}, this level of accuracy of $L^2$ error around $10^{-3}$ is enough so that the lift and drag coefficients are accurately predicted.  Moreover, the contour plots from figure \ref{contourcyl500} match the DNS very well by t=10, both for $H=0.275$ and $H=0.0688$, although at t=6 only the solution with $H=0.0688$ matches the DNS well.

To consider further the seeming lower bound on the error in the $Re=500$ tests so far, we consider additional runs with varying $H$ and $\mu$.  We show the $L^2$ errors for these tests in figure \ref{difference500}, and observe that the error seems to be bounded below by $O(\mu^{-1})$, seemingly independent of $H$ (even though the DNS solution is in the finite element space and thus 0 error is possible, just as in the $Re=100$ case).  Up to this lower bound, the DA solutions converge quickly, in particular for $H=\frac{2.2}{500}$ and $\mu=1000$ the convergence is rapid.

Overall we conclude that results for $Re=500$ are quite good.  While it appears that the error depends on $O(\mu^{-1})$, the only reason why we see this error in these tests is that the DNS was done on the same discretization as the DA.  In practice, there will also be spatial and temporal errors present, and in particular we would expect spatial error to dominate any $O(\mu^{-1})$ errors when $\mu=100$ or $1000$.

\begin{figure}[!ht]
	\centering
	$H = \frac{2.2}{48}$ \hspace{2.5in} $H=\frac{2.2}{500}$ \\
	\includegraphics[width = .49\textwidth,height=.3\textwidth,viewport=0 0 550 400, clip]{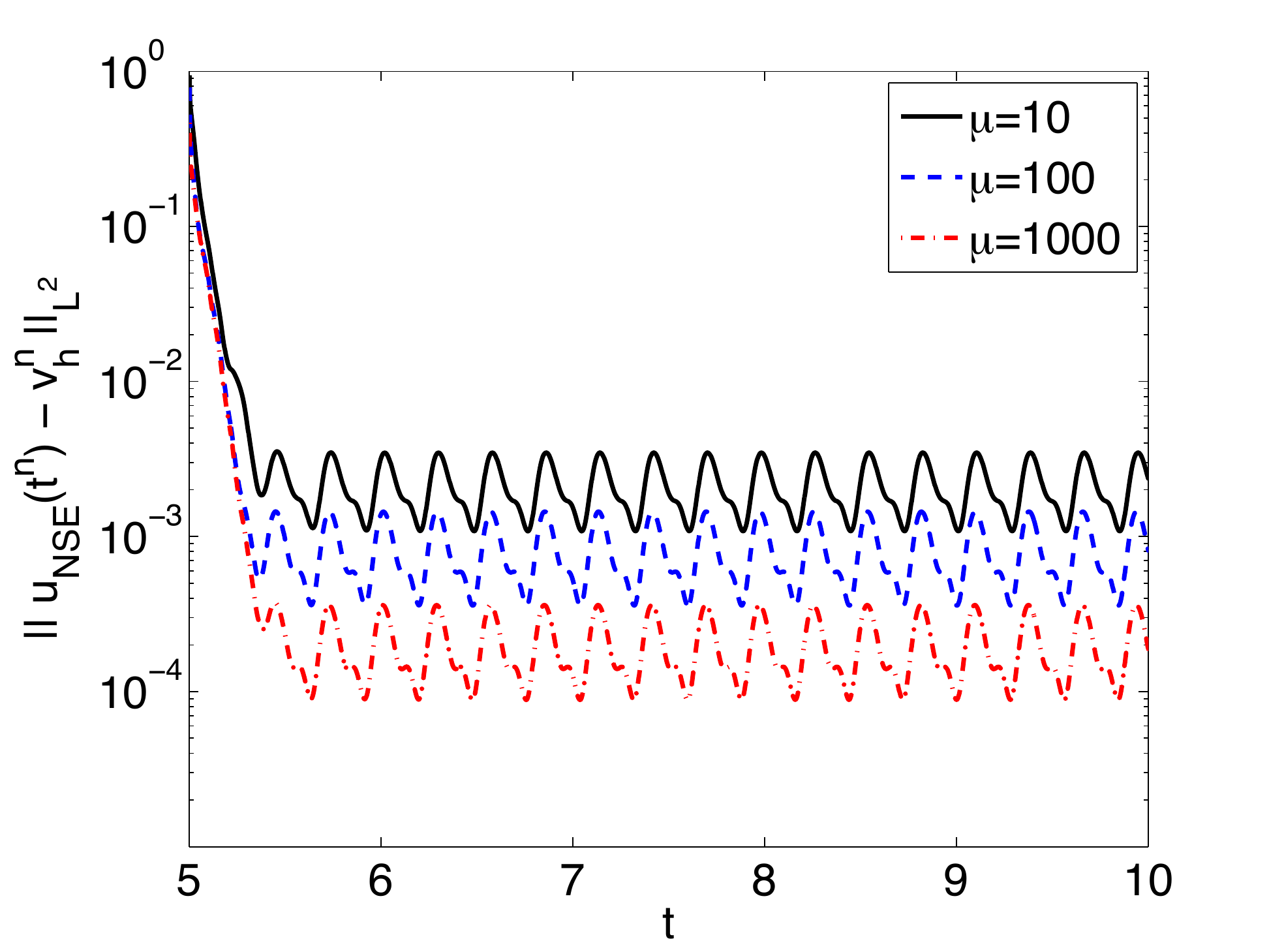}
	\includegraphics[width = .49\textwidth,height=.3\textwidth,viewport=0 0 550 400, clip]{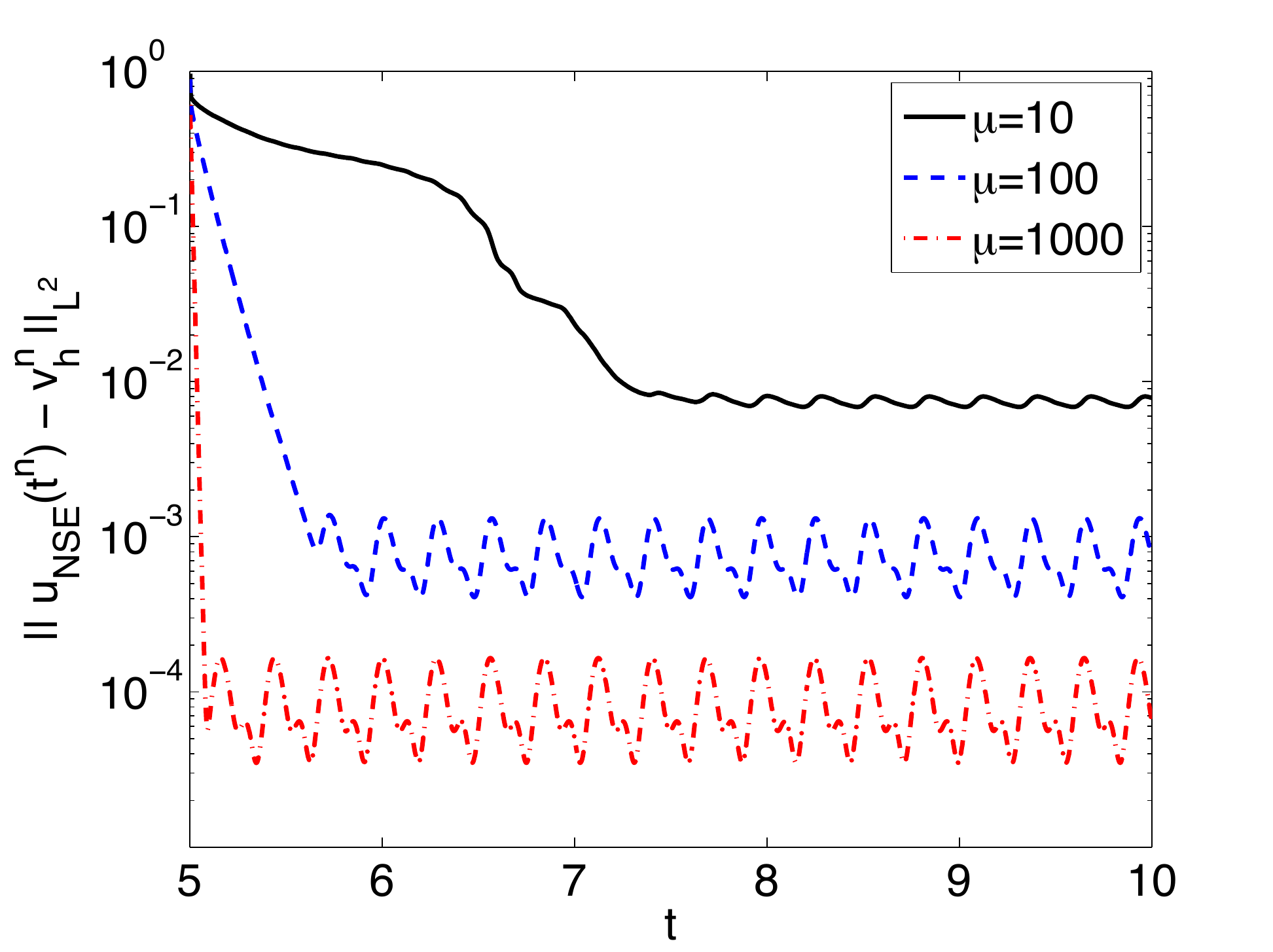}
	\caption{\label{difference500} Shown above are the $L^2$ errors versus time for $Re=500$ simulations with varying $\mu$ and $H=\frac{2.2}{48}$ (left) and $H=\frac{2.2}{500}$ (right).}
\end{figure}

\section{Conclusions and Future Directions}

We have proposed, analyzed and tested a new interpolation operator for use with DA of evolution equations.  The new operator allows for simple implementation of DA, in particular in legacy codes, as the DA can be easily implemented at the linear algebraic level independent of the rest of the discretization.  We prove that the operator has stability and accuracy properties that differ slightly from those laid out in some recent DA papers as being sufficient properties, but are still sufficient for allowing DA algorithms to be long-time stable and accurate (evidenced by our analysis and testing of DA for fluid transport and incompressible NSE).  Our numerical tests show the interpolation used in conjunction with this type of continuous data assimilation is very effective on a range of problems.

There are several important future directions to consider.  First, if given $N$ observation locations for a physical problem, then it {\color{black}may be} the case that better assimilation is possible if the locations are based on the physics of the problem instead of being picked to define a quasi-uniform interpolant.  Second, we assume herein that the observation points are fixed and are also nodes on the fine mesh; removing each of these assumptions would help the proposed methods be more applicable.  Another important problem is to consider DA algorithms with different types of boundary conditions.  We observe above that for NSE with outflow boundary conditions, the error seems to depend on $O(\mu^{-1})$ while in the case of full Dirichlet boundary conditions it does not.  Better understanding of how DA algorithms behave under other types of physical boundary conditions would be helpful to practioners.

\bibliographystyle{abbrv}
\bibliography{references,LariosBiblio}

\end{document}